\documentclass[A4paper, 11pt]{amsart}
\usepackage{amscd,amssymb,amsmath,enumerate}
\usepackage{amsmath,amsthm}
\usepackage[all]{xy}
\usepackage{graphicx}
\usepackage{graphics}
\usepackage{float}
\usepackage{latexsym}
\usepackage[all]{xy}
\usepackage{psfrag}
\usepackage{mathrsfs}  
\usepackage[normalem]{ulem}
\usepackage[colorlinks=true, linkcolor = black, citecolor = black, menucolor = black, urlcolor = black]{hyperref}
\xyoption{matrix} \xyoption{arrow}
\usepackage{parskip}
\usepackage{epstopdf}

\usepackage{tikz}
\usepackage{tkz-graph}
\usetikzlibrary{arrows}
\usetikzlibrary{backgrounds}
\usetikzlibrary{calc}
\tikzstyle{bl}=[circle, draw=white, thin,fill=black!100, scale=0.5]
\tikzstyle{cg}=[circle, draw, thin,fill=black!10, scale=0.8]
\tikzstyle{cw}=[circle, draw, thin,fill=white, scale=0.8]

\usepackage{booktabs}
\usepackage{multirow}
\usepackage{colortbl}
\usepackage{boldline}
\usepackage{cellspace}
\numberwithin{equation}{section}

\newtheorem{proposition}{Proposition}[section]
\newtheorem{theorem}[proposition]{Theorem}
\newtheorem{lemma}[proposition]{Lemma}
\newtheorem{corollary}[proposition]{Corollary}

\newtheorem{definition}[proposition]{{Definition}}
\newenvironment{defn}{\begin{definition} \rm}{\end{definition}}

\newtheorem{remark}[proposition]{{Remark}}

\newtheorem{Example}[proposition]{Example} 

\newtheorem*{Ex}{Example}

\newcommand{\cB}{{\mathcal B}}

\newcommand{\cP}{{\mathcal P}}

\newcommand{\cL}{{\mathcal L}}
\newcommand{\cM}{{\mathcal M}}

\renewcommand{\Im}{\operatorname{Im}\nolimits}
\newcommand{\Ker}{\operatorname{Ker}\nolimits}

\newcommand{\soc}{\operatorname{soc}\nolimits}

\newcommand{\val}{\operatorname{val}\nolimits}

\newcommand{\rank}{\operatorname{rank}\nolimits}

\usepackage{color}
\definecolor{alizarin}{rgb}{0.82, 0.1, 0.26}
\definecolor{candyapplered}{rgb}{1.0, 0.03, 0.0}
\definecolor{blue(pigment)}{rgb}{0.2, 0.2, 0.6}
\definecolor{darkspringgreen}{rgb}{0.09, 0.45, 0.27}
\makeatletter
\def\thm@space@setup{%
  \thm@preskip=0.7cM \thm@postskip=0.3cM
}
\makeatother

\thanks{  The second author is supported by the EPSRC through an Early Career Fellowship EP/P016294/1. The first and third authors have been partially supported by the projects UBACYT 20020160100613BA and PIP-CONICET 112–201501–
00483CO, and they also thank the Department of Mathematics of the University of Leicester. The third author is a research member of CONICET (Argentina).}

\begin{document}

\title[On the Hochschild cohomology of gentle and Brauer graph algebras]
{On the Lie algebra structure of the first Hochschild cohomology of gentle algebras and Brauer graph algebras}

\author[Chaparro]{Cristian Chaparro}
\address{Cristian Chaparro\\
Departamento de Matem\'atica \\
 FCEyN Universidad de Buenos Aires \\ 
 Pabellon I - Ciudad Universitaria \\
1428 - Buenos Aires \\ 
  Argentina}
\email{cchaparro@dm.uba.ar}

\author[Schroll]{Sibylle Schroll}
\address{Sibylle Schroll\\
Department of Mathematics \\
University of Leicester \\
University Road  \\
Leicester LE1 7RH, UK
}
\email{schroll@le.ac.uk}

\author[Solotar]{Andrea Solotar}
\address{Andrea Solotar\\
Departamento de Matem\'atica \\
 FCEyN Universidad de Buenos Aires \\ 
 Pabellon I - Ciudad Universitaria \\
1428 - Buenos Aires \\ 
  Argentina}
\email{asolotar@dm.uba.ar}



\dedicatory{Dedicated to Michel Brou\'e who has done so much for algebra as a subject. }

\begin{abstract}
In this paper we determine the first Hochschild homology and cohomology with different coefficients for gentle algebras and we give a geometrical interpretation of these (co)homologies using the ribbon graph of a gentle algebra as defined in \cite{S}. 
We give an explicit description of the Lie algebra structure of  the first Hochschild cohomology  of gentle  and Brauer graph algebras  (with multiplicity one) based on trivial extensions of gentle algebras and we show how the Hochschild cohomology is encoded in the Brauer graph.  In particular, we show that except in one low-dimensional case, the resulting Lie algebras are all solvable. \end{abstract}

\maketitle


\section{Introduction}

Gentle algebras have gained much traction in recent years due to their pivotal role in several different areas of mathematics. They appear in the guise of Jacobian algebras of quivers with potential from unpunctured marked surfaces in cluster theory \cite{Labardini, ABCP, DRS, Amiot, AG}, they play a role in certain gauge model theories in theoretical physics in the work of Cecotti, where they appear as Jacobian algebras of quivers and superpotentials 
corresponding  to a Gaiotto $A_1$-theory with only irregular punctures and at least one such puncture  
\cite{Cecotti} and they are an important element in homological mirror symmetry of 2-manifolds, in that the derived category of a differential graded smooth gentle algebra is equivalent to the partially wrapped Fukaya category of a surface with stops \cite{HKK, LP}.  In the trivially graded case these categories are equivalent to the bounded derived category of the  (ungraded) gentle algebra and based on the ribbon graph of a gentle algebra defined in \cite{S}, a geometric surface model of this category (coinciding with the model underlying the partially wrapped Fukaya category) has been given in \cite{OPS}. 

It is clear that the derived category of a gentle algebra is a key tool in this setting. By the work of Happel there is a fully faithful embedding of the  bounded derived category of a finite dimensional algebra into the stable module category of its repetitive algebra. The repetitive algebra is closely linked to the trivial extension of the algebra.  Therefore, understanding the trivial extension of an algebra $A$ works towards understanding the derived category of $A$.


In this paper we take this point of view by  determining the Lie algebra structure of the first Hochschild cohomology groups of both gentle algebras and their trivial extensions, which cover all Brauer graph algebras with trivial multiplicities.  We will exhibit an example of a Brauer graph algebra which can be obtained both as a trivial extension of a gentle algebra with finite global dimension and also as a trivial extension of a
gentle algebra with infinite global dimension, showing in this way that the global dimension is an invariant which in fact plays no role in the Lie algebra structure of the first Hochschild cohomology space of the trivial extension.

The Hochschild cohomology of an algebra $A$ is an important tool attached to the algebra. Given a $K$-algebra $A$, the graded vector space $HH^*(A)$ is a derived invariant \cite{Rickard}, as is the Hochschild homology $HH_*(A)$. However, $HH^*(A)$ has a richer structure:
	\begin{itemize}
		\item it is a graded commutative algebra via the cup product;
		\item it is endowed with the Gerstenhaber bracket:
		\[[-,-]:HH^m(A)\otimes HH^n(A) \rightarrow HH^{m+n-1}(A)\text{, for all } m,n\geq 0\]
		such that $(HH^*(A),[-,-])$ is a graded Lie algebra.
	\end{itemize}
These structures are related by the Poisson identity, meaning that the bracket is a graded biderivation with respect to the cup product.

All this together endows $HH^*(A)$ with a  {\em Gerstenhaber algebra structure}, which is also a derived invariant, see \cite{Keller}.
In particular, the first Hochschild cohomology space $HH^1(A)$ which is isomorphic to the quotient of the  derivations of $A$ modulo the inner derivations of $A$, 
becomes a Lie 
algebra whose bracket is induced by the commutator of derivations, and for all
$n\in \mathbb{N}$, $HH^n(A)$ is a Lie $HH^1(A)$-module.

In characteristic $p$, $HH^1(A)$ is a restricted Lie algebra, and this is again a derived invariant \cite{Zimmermann}. When working in characteristic different from $2$, the squaring maps $sq_i:HH^{2i}(A)\rightarrow HH^{4i-1}(A)$ are defined by $a\mapsto \frac{1}{2}[a,a]$, for $i\geq 1$. This gives $HH^*(A)$ the structure of  a {\em strict Gerstenhaber algebra} \cite{Gerstenhaber}. Note that in particular, $sq_1:HH^2(A)\rightarrow HH^3(A)$ provides possible obstructions to infinitesimal deformations.

In general, determining even the first Hochschild cohomology group of an algebra requires a fair amount of more or less involved calculations. Moreover, the structural invariants related the Gerstenhaber bracket are not easy to compute, since they are defined in terms of the Bar resolution. The cup product is better understood, since it can be expressed in terms of a diagonal $\Delta_A:P_A\rightarrow P_A \otimes_A P_A$, where $P_A$ is any projective resolution of $A$ as $A$-bimodule. Also, it coincides with the Yoneda product of extensions, which is sometimes useful for computations.

In spite of the existence of results obtained by Stasheff \cite{Stasheff}, Schwede \cite{Schwede} and Keller \cite{Keller}, the Gerstenhaber bracket remains elusive. 
Several examples of computations are known, see for instance \cite{Redondo-Roman}, where the authors prove that in characteristic different from $2$ the bracket of the class of an $n$ cocycle with the class of an $m$ cocycle is null for $n,m$ even and positive. Moreover, some methods have been developed recently by Negron and Witherspoon \cite{Negron-Witherspoon} and \cite{Negron-Witherspoon2}, Volkov \cite{Volkov},  Negron, Volkov and Witherspoon \cite{NVW} and, via a different approach, by Su\'arez-\'Alvarez \cite{Suarez}.

In this article we will give a description of the Lie algebra structure of the first Hochschild cohomology space of gentle algebras and any Brauer graph algebra with multiplicity $1$, using the fact that the latter algebras are 
trivial extensions of gentle algebras. In particular, we will show how the ribbon graph of a gentle algebra and the Brauer graph of the Brauer graph algebra encode the bracket.
Furthermore, we show the following. 

{\bf Theorem.} (see Thm. \ref{KroneckerTh}) {\it
		Let $A$ be a gentle algebra $K$-algebra, and let $B$ be the Brauer graph algebra isomorphic to $TA$. Suppose ${\rm char }  K\neq 2$. Then
		\begin{enumerate}
			\item $HH^1(A)$ is solvable if and only if $A$ is not the Kronecker algebra. 
			\item $HH^1(B)$ is solvable if and only if $B$ is not isomorphic to the trivial extension of  the Kronecker algebra.  
\end{enumerate}
}

The aforementioned  correspondence allows us to use results by Cibils, Marcos, Redondo and Solotar \cite{CMRS} and Cibils, Redondo and Saor\'\i n \cite{CRS} that provide an explicit decomposition of the first cohomology space of a trivial extension algebra $A$ in terms of the centre of the algebra, its Hochschild  cohomology in degree $1$,  its Hochschild  homology in degree $1$ and another space denoted $Alt_A(DA)$ that we will recall in Section \ref{Alt}. We will also use the description of the bracket given by Strametz \cite{Str}.
More precisely, given a gentle algebra $A$, the trivial extension $TA=A\ltimes DA$ of $A$ by its minimal cogenerator $DA$ is a symmetric special biserial algebra \cite{PS,R, SJ} and hence a Brauer graph algebra. There is an embedding of $HH^1(A)$ into $HH^1(TA)$ --see \cite{Str}--, given by the map sending the class  of a $K$-linear derivation $\varphi_{A,A}:A \rightarrow A$ into $(-D(\varphi_{A,A}),0,\varphi_{A,A},0)$ in 
$HH^1(TA)\cong Z(A)\oplus HH_1(A)^*\oplus HH^1(A)\oplus Alt_A(DA)$ where $HH_1(A)^*$ denotes the $K$-dual of $HH_1(A)$ and we identify the centre $Z(A)$ 
with $Hom_{A-A}(DA,DA)$ and $D(\varphi_{A,A}):DA\rightarrow DA$ is such that  
$D(\varphi_{A,A})(f)=f\circ\varphi_{A,A}$. 

Gentle algebras and Brauer graph algebras are -- as it is well-known-- algebras of tame representation type, their first Hochschild cohomology space is proved here to be a solvable Lie algebra. It is already known that this happens for other tame algebras, such as the special biserial algebras considered in \cite{MNPRS} and the toupie algebras with no branches of length one --that is, exactly the tame ones--, see \cite{ALS}.  In view of these results, in an earlier version of the paper we posed the following question:

{\bf Question:} Is it true that, except in some low dimensional cases,  the first Hochschild cohomology space of any finite dimensional algebra of tame representation type is a solvable Lie algebra?

Through subsequent work \cite{ER,RSS}, further examples of non-solvable first Hochschild cohomology as  a Lie algebra both for tame and wild algebras have appeared. However, 
so far the only contribution to the first Hochschild cohomology as a Lie algebra of a finite dimensional algebra are solvable Lie algebras and Lie algebras of type $A$, that is $\mathfrak{sl}(m,K)$ and $\mathfrak{gl}(m,K)$.

We will assume that all algebras are indecomposable and finite dimensional. 


{\bf Acknowledgement:}  We thank the referee for the many comments and suggestions which have considerable improved the exposition of the paper. We also thank them for pointing out a mistake in Theorem \ref{HH^1} (4) in an earlier version of the paper.
\section{Background}

Let $A$ be a finite dimensional associative unital algebra, denote by $DA$ its $K$-linear dual endowed with 
the usual $A$-bimodule structure. In what follows, we will write $V^*$  to denote the $K$-linear dual of the vector space $V$.
Following \cite{CMRS}, we denote  by  $Alt_A (DA)$ the set of skew-symmetric bilinear forms $\alpha$ over $DA$ such that $\alpha(fa, g) = \alpha (f, ag)$ for all $a\in A$ and $f,g \in DA$.
Recall from \cite{CMRS} that there is an isomorphism of vector spaces
\[HH^1(TA) \cong Z(A) \oplus HH_1(A)^* \oplus HH^1(A) \oplus Alt_A(DA)\] and that, in particular, we have  $HH^1(TA) \neq 0$, since $Z(A)$ is never zero.  
The Hochschild homology and cohomology of monomial algebras have been widely studied. In \cite{CRS}, the authors provide the dimensions of $HH_1(A)$
and $Alt_A(DA)$ in terms of the quiver of $A$ when $A$ is monomial. In this article we give a description of these spaces and provide bases for the four summands that are particularly adapted to make a link with the ribbon graph of the gentle algebra.

We will now briefly recall the definition of gentle algebras and that of Brauer graph algebras. 

\begin{defn}
An algebra $A$ is called \emph{gentle} if it is Morita equivalent to an algebra $KQ/I$ such that 
\begin{enumerate}
\item for every vertex $v \in Q_0$, there are at most two arrows arriving and two arrows starting at $v$,
\item for every arrow $a \in Q_1$, there exists at most one arrow $b \in Q_1$ such that $ab \notin I$ and at most one arrow $c \in Q_1$ such that $ca \notin I$,
\item for every arrow $a \in Q_1$, there exists at a most one arrow $b \in Q_1$ such that $t(a) = s(b)$  and $ab \in I$ and at most one arrow $c \in Q_1$ such that $t(c) = s(a)$ and $ca \in I$,
\item $I$ is generated by paths of length 2.
\end{enumerate}
\end{defn}

Brauer graph algebras are defined by a ribbon graph, referred to as the Brauer graph of the Brauer graph algebra, with a vertex decoration by strictly positive integers. In this paper all the Brauer graph algebras are trivial extensions of gentle algebras and therefore all vertex decorations are equal to 1, see \cite{S}. Therefore we can omit the vertex decorations altogether (implicitly assuming that they are equal to 1) and we only give the definition of Brauer graph algebras in this case. 

\begin{defn}
Let $\Gamma$ be a finite ribbon graph (that is a finite unoriented graph with at each vertex, a cyclic ordering of the edges incident with that vertex). The Brauer graph algebra $\Lambda_\Gamma$ is the algebra $KQ/I$
where the quiver $Q$ is such that 
\begin{enumerate}
\item the vertices of $Q$ correspond to the edges of $\Gamma$
\item there is an arrow from vertex $v$ to vertex $w$ in $Q$ if the corresponding edges $E_v$ and $E_w$ in $\Gamma$ are incident with the same vertex $m$ of $\Gamma$ and if $E_w$ directly follows $E_v$ in the cyclic ordering of the edges incident with $m$.
\end{enumerate}
The definition of the ideal $I$ is  based on the observation that for every vertex $v$ in $Q$, there are two cycles $C_{v,m}$ and $C_{v,n}$ in $Q$ starting at vertex $v$, which are induced by the cyclic orderings of the edges incident with the vertices $m$ and $n$ of the edge $E_v$ of $\Gamma$ corresponding to $v$. If one of the vertices  of $E_v$ is of valency 1 then corresponding cycle is trivial. 
The ideal $I$ is generated by the following relations  
\begin{enumerate}
\item 
for every $v \in Q_0$ such that both cycles $C_{v,m}$ and $C_{v,n}$  are non-trivial, we have the relation $C_{v,m}-C_{v,n}$,
\item for every $v \in Q_0$, we have the relations $C_{v,m}\alpha$ and $C_{v,n}\beta$ where $\alpha $ is the first arrow of $C_{v,m}$ and $\beta$ is the first arrow of  $C_{v,n}$,
\item any path of length two,           which is not a subpath of any $C_{v,m}$ for all $v \in Q_0$ and all $m \in \Gamma$, is a relation.
\end{enumerate}
\end{defn}

Recall from \cite{S} that if $\Lambda$ is a Brauer graph algebra  with multiplicity function identically equal to one, then there exists a gentle algebra $A$ such that $\Lambda = TA$. 
 
 We recall this in more detail now. For that 
let $p$ be a path in $Q$, if $p = p_2  e_v p_1$ for  possibly trivial paths $p_1$ and $ p_2$ then we write $v \in p$.  

Let $A$ be a gentle algebra and let $\cM$ be a basis of $  \soc_{A^e} A$ where $A^e = A \otimes_K A^{op}$. Since $A$ is monomial, a basis of  $\soc_{A^e} A$ is given by images of maximal paths $p$ in $Q$ which are maximal in the sense that for any arrow $a \in Q_1$, both $a p $ and $p a$ are in $I$ . Denote by $V_0$ the set of vertices $v$ in $Q$ such that one of the 
following holds 
\begin{itemize}
 \setlength\itemsep{-.5em}
\item  $v$ is a source of a single arrow in $Q$,
\item $v$ is a sink of a single arrow in $Q$, 
\item $v$ is the start of an arrow $a$ and target of an arrow $b$ where $ab \not \in I$ and there is no other arrow starting or ending at $v$. 
\end{itemize}
Set $\overline{\cM} = \cM \cup \{e_v \mid v \in V_0 \}$. Note that  every vertex of $Q$ appears exactly twice in $\overline{\cM}$. 
Given an element $m$ in $\cM$ denote by $\beta_m$ a new arrow such that $s(\beta_m) = t(m)$ and $t(\beta_m) = s(m)$. Then $Q \cup \{\beta_m \vert m \in \cM \}$ is the quiver of $TA$. We recall the following definitions from \cite{S} and \cite{SchrollSurvey}. 

\begin{definition} {\rm 
We define the \emph{ribbon graph of a gentle algebra $A = KQ/I$}  to be the graph with vertices given by
$\overline{\cM}$ and for $m, m' \in \overline{\cM}$ there is an edge between $m$ and $m'$ if there exists a vertex $v \in Q$ such that $v \in m$ and $v \in m'$. The cyclic ordering of the edges at vertex $m$ is given by the cyclic closure of the linear order of the vertices induced by $m$.

That is if $m = e_{v_1} a_1 e_{v_2} a_2 \ldots a_n e_{v_{n+1}}$ then the corresponding edges in the ribbon graph are cyclically ordered $E_{v_1} < E_{v_2} < \cdots < E_{v_{n+1}} < E_{v_1}$. 

The \emph{marked ribbon graph $\Gamma_A$ of $A$} is given by marking two successive edges of the ribbon graph of $A$ corresponding to $s(\beta_m)$ and $t(\beta_m)$. We often write $\Gamma$ for $\Gamma_A$ if the algebra is clear from the context and we write $\Gamma_0$ for the vertices of $\Gamma$ and $\Gamma_1$ for the edges of $\Gamma$. 
}
\end{definition}

It directly follows from the definitions that the edges of $\Gamma_A$ are in bijection with the vertices of $Q$. 

\begin{Example}
{\rm Let $A = K( \xymatrix{ e_1 \ar@<-.5ex>[r] \ar@<.5ex>[r] & e_2 } )$ be the Kronecker quiver, then the marked ribbon graph $\Gamma_A$ is given by}
\vspace{-0.4cm}
	\begin{figure}[H]
		\centering {\footnotesize
			\tikzstyle{help lines}+=[dashed]
			\begin{tikzpicture}[auto, thick,scale=1]+=[dashed]
				\clip (-0.5,-.9) rectangle (2.5,0.8);
				\node[cg] (1) at (0:0){\,\,\,};
				\node[cg] (2) at ($(1)+(0:2)$){\,\,\,};
				\node () at ($(1)-(0:0.4)$){$\times$};
				\node () at ($(2)-(0:0.4)$){$\times$};
				
				\draw[-] (1) to [out=35, in=145] node {$e_1$} (2);
				\draw[-] (1) to [out=-35, in=-145] node [below] {$e_2$} (2);
				\draw[->] (1)+(25:.5cm) arc (25:-25:.5cm);
				\draw[->] (2)+(140:.4cm) arc (140:-140:.4cm);
			\end{tikzpicture}
			\caption{ Marked ribbon graph of the Kronecker algebra. For illustration purposes we have also included the arrows of the quiver of the algebra.}}
	\end{figure}
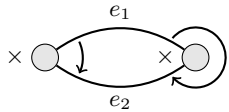
\end{Example}

\section{Components of the first cohomology group of trivial extensions of gentle algebras}

From now on assume that $A = KQ/I$ is a gentle algebra where $I$ is generated by a minimal set $R$ of paths of length $2$, $Q$ is a connected quiver and  $A$ is not a point or a single loop. The cases of a point and a single loop will be treated separately in Section~\ref{tpc}. 

Let $\cB$ be the set of paths of $Q$ which do not contain any path of $R$ as a subpath. Note that $\cB$ is a $K$-basis of $A$.  We will interpret the first and second summands of the decomposition of $HH^1(TA)$ in terms of parallel paths (see \cite{Str}); that is, paths which share source and target.

\begin{definition}\label{def::spaces}{\rm 
		Two paths $p, q$ of $Q$ are called {\em parallel} if $s(p)=s(q)$ and $t(p)=t(q)$. If $X$ and $Y$ are sets of paths of $Q$, the {\em set $X\vert \vert Y$  of parallel paths} is formed by the pairs $(p,q)$ in $X\times Y$ such that $p$ and $q$ are parallel paths. We denote by $K(X\vert \vert Y)$ the vector space with basis the set $X\vert \vert Y$.}
\end{definition}

 Given a path $p$ in $Q$ and $(a,q)\in Q_1\vert \vert \cB$, we denote by $p^{(a,q)}$ the sum of all nonzero paths where each term in the sum is obtained by replacing exactly one appearance of the arrow $a$ in $p$ by the path $q$. If the path $p$ does not contain the arrow $a$ or if every replacement of $a$ in $p$ by $q$ is not a path in $\cB$, we set $p^{(a,q)}=0$. The function $\chi_\cB:\bigcup Q_n\rightarrow \{0,1\}$ denotes the indicator function which associates $1$ to each path $p\in \cB$ and $0$ to $p\notin \cB$.
 Given a graph $G$  (or a quiver $Q$) and $v$ a vertex in $G$ (or $Q$), denote by $\val (v)$ the number of edges (or arrows) incident with $v$. Then for any graph $G$, we have $2  \vert G_1\vert = \sum_{v \in G_0} \val (v)$.

\begin{proposition} \cite[Corollaire  2.2.1.9]{Str}\label{prop::str}
	Let $A$ be a monomial algebra. The Hochschild cohomology of $A$ in low degrees is the cohomology of the cochain complex
	\[0\longrightarrow K(Q_0\vert \vert \cB)\xrightarrow{d^0} K(Q_1\vert \vert \cB) \xrightarrow{d^1} K(R\vert \vert \cB)\rightarrow...\]
	where 
	\begin{equation}
	d^0(e,p)=\sum_{a\in Q_1e} \chi_\cB(ap)(a,ap)-\sum_{a\in eQ_1} \chi_\cB(p a)(a,pa),
	\label{d0}
	\end{equation}
	\begin{equation}
	d^1(a,p)=\sum_{q\in R} (q,q^{(a,p)}).
	\label{d1}
	\end{equation}
\end{proposition}

\subsection{Centre of a gentle algebra }
For the convenience of the reader we briefly recall the structure of the centre of a gentle algebra. 
\begin{lemma}\label{basisZ(A)}
Let $A = KQ/I$ be a gentle algebra. The set of the following elements is a basis of $Z(A)$.
\begin{enumerate}
	\item $\sum_{e\in Q_0} e$,
	\item $p=a_n \ldots a_1$ in $\cB$ such that  $t(a_n) = s(a_1)$,   $a_1 a_n \in I$ and $\val (t(a_n))= 2$. 
\end{enumerate}
\end{lemma}

\begin{proof}
	Let $p$ be a cycle in $\cB$ with length $n\geq 1$ and $d^0$ as in \eqref{d0}, it is clear that if $(e,p)\in Q_0\vert \vert \cB$ and $\val(e)=2$, then $(e,p)\in \Ker d^0$.
	
	Note that  if $(e,p)\in Q_0\vert \vert \cB$ with  $\val (e)> 2$ and $n\geq1$, then  
	\begin{equation*}
	d^0(e,p)= \left\lbrace
	\begin{array}{ll}
	(b_2,b_2p) & \text{ if } b_2p\in \cB \text{ and }  \val(e)=3,\\
	-(b_1,p b_1)& \text{ if } p b_1\in \cB \text{ and }  \val(e)=3,\\
	(b_2,b_2p)-(b_1,p b_1)& \text{ if }  b_2p, p b_1\in \cB \text{ and } \val(e)=4,
	\end{array}\right.
	\end{equation*}
where $b_1$ and $b_2$ are arrows which do not appear in $p$ and at least one of them exists.
\begin{figure}[H]
	{\footnotesize
	\centering \tikzstyle{help lines}+=[dashed]
	\begin{tikzpicture}[auto, thick,scale=0.9]+=[dashed]
	\node (0) at (0:0) {$\bullet$};
	\node (1) at ($(0)+(0:2)$) {$e$};
	\node (2) at ($(1)+(0:2)$) {$\bullet$};	
	\node () at ($(2)+(0:.5)$) {$...$};
	\node () at ($(0)-(0:.5)$) {$...$};
	\draw[->,dashed] (1) to [out=130, in=50, looseness=10] node {$p$}(1);
	\draw[->] (0) to node {$b_1$} (1);
	\draw[->] (1) to node {$b_2$} (2);
	\draw[style=help lines] (1) +(200:.45cm) arc (200:350:.45cm);
	\draw[style=help lines] (1) +(55:.55cm) arc (55:130:.55cm);
	\end{tikzpicture}
	\caption{Locally $Q$ has this shape if both arrows $b_1$ and $b_2$ exist.}}
\end{figure}
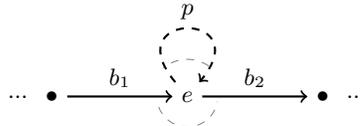
Now, suppose that $(e,p)\neq(e',p')$ and both satisfy the above conditions, then the summands of $d^0(e,p)$ do not appear amongst the summands of $d^0(e',p')$. Indeed, if $(b_1,pb_1)=(b'_1,p'b'_1)$, then $b_1=b'_1$ and $p=p'$, hence $e=s(p)=s(p')=e'$, which is a contradiction (similarly if  $(b_2,b_2p)=(b'_2,b'_2p')$); if $(b_2,b_2p)=(b'_1,p'b'_1)$, then we have $b_2=b'_1$ and $b_2p=p'b_2$, thus $p'=b_2q$ and $p=q'b_2$ for some paths $q, q'\in \cB_{n-1}$, but $b_2qb_2=b_2q'b_2$ is not possible in $A$.

On the other hand, note that $d^0(e,e)=\sum_{b\in Q_1e}(b,b)-\sum_{b\in eQ_1}(b,b)=0$ if and only if  $\val(e)=2$ and $e$ is the vertex of a loop. Finally, if $Q$ has at least two vertices and is connected, then the associated matrix to the system $d^0(e,e)=0$ has column vectors with two inputs different to zero and with values $1$ and $-1$, or the zero vector that correspond to a loop, and row vectors with at most four inputs nonzero, with at most two inputs equal to $-1$ and with at most two inputs equal to $1$. Furthermore, the rank is $|Q_0|-1$ and
\[\sum_{e\in Q_0}d^0(e,e)=\sum_{e\in Q_0}\left( \sum_{b\in Q_1e}(b,b)-\sum_{b\in eQ_1}(b,b)\right) =0.\]
\end{proof}
The following local configurations in the ribbon graph of $A$ each contribute an element to a basis of  $Z(A)$.  
\begin{figure}[H]
			\centering
			\begin{tikzpicture}[xscale=1.1, auto, thick]+=[dashed]
			\clip (-0.5,-2.5) rectangle (9.4,1.6);
			\node (1) at (0:0){$Id_A$};
			{\normalsize \node () at (-90:1.9){(1)};}
			{\footnotesize
			\node[cg] (2) at ($(1)+(0:2.5)$){$a_1$};
			\node () at ($(2)-(0:0.4)$){$\times$};
			\node () at ($(2)+(0:0.8)$){$a_1$};
			\draw[-] (2) to [out=40, in=-40, looseness=25] node {}(2);
			\draw[->] (2) +(33:.6cm) arc (33:-33:.6cm);
			{\normalsize \node () at ($(2)+(0:0.9)+(-90:1.9)$){(2.a)};}
			\node[cg] (3) at ($(2)+(0:5)$){$p$};
			\node () at ($(3)+(0:.6)$){$\times$};
			\draw[-] (3) to [out=40, in=-40, looseness=25] node {}(3);
			\draw[-] (3) to node {} ($(3)+(120:1.7)$);
			\draw[-] (3) to node {} ($(3)+(150:1.55)$);
			\draw[-] (3) to node {} ($(3)+(210:1.55)$);
			\draw[-] (3) to node {} ($(3)+(240:1.7)$);
			
			\node () at ($(3)+(280:1.1)$){$a_1$};
			\draw[->] (3)+(315:.8cm) arc (315:245:.8cm);
			\node () at ($(3)+(225:1.1)$){$a_2$};
			\draw[->] (3)+(235:.8cm) arc (235:215:.8cm);
			\draw[-,dashed] (3)+(205:.8cm) arc (205:155:.8cm);
			\node () at ($(3)+(137:1.15)$){$a_{n-1}$};
			\draw[->] (3)+(145:.8cm) arc (145:125:.8cm);
			\node () at ($(3)+(80:1.1)$){$a_{n}$};
			\draw[->] (3)+(115:.8cm) arc (115:45:.8cm);
			{\normalsize \node () at ($(3)+(-90:1.9)$){(2.b)};}
			}
	\end{tikzpicture}
\caption{Local configurations in the marked ribbon graph of a gentle algebra, each giving rise to a basis element of $Z(A)$.} 
	\label{fig:config 1}
\end{figure}
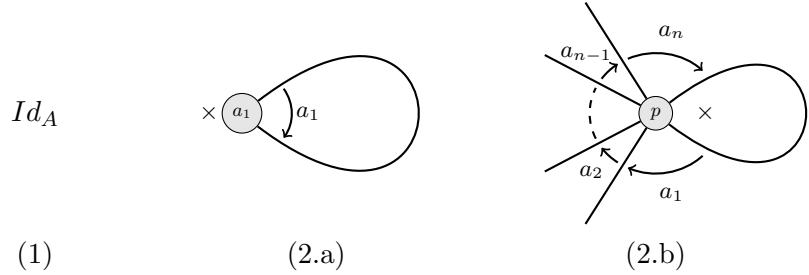

\subsection{$HH^1(A)$ of a gentle algebra $A$}

Let $A = KQ/I$ be a gentle algebra. In the case of a gentle algebra which is given by a Jacobian algebra of a quiver with potential from a triangulation of an unpunctured surface, in \cite{Valdivieso}  the Hochschild cohomology and a  geometric interpretation have been given. Furthermore, in addition to \cite{Ladkani1}, the second author has had a private communication from S. Ladkani detailing  the Hochschild cohomology of a gentle algebra  \cite{Ladkani2}. 

The case of a  gentle algebra consisting of a single loop with relation or a single point are being treated as a separate case in Section~\ref{tpc}, this is in order to keep the statements as concise as possible. 

The following theorem is the main result of this section. The interpretation of the different cases in Theorem \ref{HH^1}  in terms of the ribbon graph of the gentle algebra is given in  Figures \ref{fig:config 4}, \ref{fig:config 5}, \ref{fig:config 6} and the notion of  {\em fundamental cycle} is defined in \ref{subsection::fundamental cycles} below. In the following, we use the notation introduced in Definition~\ref{def::spaces}.

\begin{theorem}\label{HH^1} Let $A = KQ/I$ be a gentle algebra with basis of paths $\cB$. The following elements $(a,q) \in K(Q_1 \vert \vert   \cB)$ give a basis of $HH^1(A)$. 
	\begin{enumerate}
		\item $(a,q)$ such that $a$ does not appear in $q$ and $a$ is not in any relation in $I$. We say that $a$ is a shortcut for $q$. 
		\item $(a,q)$ such that $q = q_2 a q_1$ with $q_1, q_2 \notin Q_0$ and $a$ is not in any relation in $I$. We say that $q$ is a deviation via $a$. 
		\item $(a,a)$ where we choose exactly one arrow $a$ in each fundamental cycle in $Q$.
		\item If ${\rm char } K=2$ in addition to the above, $(a,s(a))$ such that $a$ is a loop.
	\end{enumerate}
\end{theorem}
In terms of the ribbon graph of the gentle algebra $A$, we have the following local configurations:
\vspace{-0.2cm}
\begin{figure}[H]
	{\tiny
		\centering
		\begin{tikzpicture}[auto, thick, xscale=1.1]
		\clip (-1,-1.7) rectangle (4,1.6);
		\node[cg] (0) at (0:0){\,\,\,};
		\node[cg] (1) at ($(0)+(0:2.5)$){$q$};
		\draw[-] (0) to [out=60, in=120, looseness=1.1] node   {} (1);
		\draw[-] (1) to [out=240, in=300, looseness=1.1]  node  {} (0);
		\node () at ($(0)-(0:.5)$){$\times$};
		\node () at ($(1)-(0:.5)$){$\times$};
		\draw[-,dashed] (0) to node  {} ($(0)+(120:1.55)$);
		\draw[-,dashed] (0) to node  {} ($(0)+(-120:1.55)$);
		\draw[->] (0)+(45:.4cm) arc (45:-45:.4cm);
		\node () at ($(0)+(0:0.6)$) {$a$};
		\draw[-] (1) to node {} ($(1)+(60:1.55)$);
		\draw[-] (1) to node {} ($(1)+(30:1.55)$);
		\draw[-] (1) to node {} ($(1)+(-30:1.55)$);
		\draw[-] (1) to node {} ($(1)+(-60:1.55)$);
		
		\node () at ($(1)+(95:1.1)$){$q_1$};
		\draw[->] (1)+(125:.8cm) arc (125:65:.8cm);
		\node () at ($(1)+(45:1.1)$){$q_2$};
		\draw[->] (1)+(55:.8cm) arc (55:35:.8cm);
		\draw[-,dashed] (1)+(25:.8cm) arc (25:-25:.8cm);
		\node () at ($(1)+(-42:1.25)$){$q_{n-1}$};
		\draw[->] (1)+(-35:.8cm) arc (-35:-55:.8cm);
		\node () at ($(1)+(-95:1.1)$){$q_{n}$};
		\draw[->] (1)+(-65:.8cm) arc (-65:-125:.8cm) ;
		\end{tikzpicture}
	}
	\caption{Where on the left hand side the marking $\times$ is between any two edges except inside the two-cycle. This corresponds to a shortcut as in Theorem~\ref{HH^1} (1) with $q=q_n...q_1$ and $q_i\in Q_1$ for $1\leq i \leq n$.}
	\label{fig:config 4}
\end{figure}
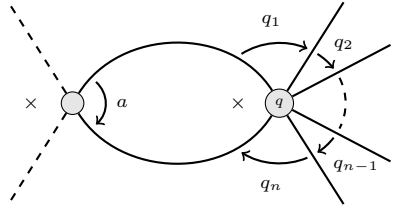
\vspace{-0.2cm}
\begin{figure}[H]
	{\tiny
		\centering
		\begin{tikzpicture}[auto, thick, scale=1.2]
		\clip (-3,-1.2) rectangle (3,1.6);
		\node[cg] (0) at (0:0){$q$};
		\node () at ($(0)+(-90:.5)$){$\times$};
		\Loop[dist=4.5cm,dir=EA,style={-}](0)
		\Loop[dist=4.5cm,style={-}](0);
		
		\draw[->] (0) +(38:1.2cm) arc (38:25:1.2cm);
		\node () at ($(0)+(31:1.4)$){$q^1_2$};
		\draw[-] (0) to node {}($(0)+(23:1.8)$);
		\draw[->] (0) +(21:1.2cm) arc (21:10:1.2cm);
		\node () at ($(0)+(16:1.4)$){$q^2_2$};
		\draw[-, line width=0.03mm, style={dashed}] (0) to node {}($(0)+(8:1.5)$);
		\draw[style={dashed},line width=0.15mm] (0) +(20:1.7cm) arc (20:-20:1.7cm);
		\draw[-] (0) to node {}($(0)+(-23:1.8)$);
		\node () at ($(0)+(-30:1.4)$){$q^n_2$};
		\draw[->] (0) +(-25:1.2cm) arc (-25:-38:1.2cm);
		
		\draw[->] (0) +(218:1.2cm) arc (218:205:1.2cm);
		\node () at ($(0)+(211:1.4)$){$q^1_1$};
		\draw[-] (0) to node {}($(0)+(203:1.8)$);
		\draw[->] (0) +(201:1.2cm) arc (201:190:1.2cm);
		\node () at ($(0)+(196:1.4)$){$q^2_1$};
		\draw[-, line width=0.05mm, style={dashed}] (0) to node {}($(0)+(188:1.5)$);
		\draw[style={dashed},line width=0.15mm] (0) +(200:1.7cm) arc (200:160:1.7cm);
		\draw[-] (0) to node {}($(0)+(157:1.8)$);
		\node () at ($(0)+(150:1.4)$){$q^m_1$};
		\draw[->] (0) +(155:1.2cm) arc (155:142:1.2cm);
		
		\node () at ($(0)+(90:1.4)$){$a$};
		\draw[->] (0) +(138:1.2cm) arc (138:42:1.2cm);	
		\end{tikzpicture}
	}
	\caption{This local configuration corresponds to a deviation $q=q_2^n...q_2^1aq_1^m...q_1^1$ via $a$  as in Theorem~\ref{HH^1} (2) and  $q_1^i, q_2^j$ arrows for $1\leq i\leq m$, $1\leq j\leq n$.}
	\label{fig:config 5}
\end{figure}
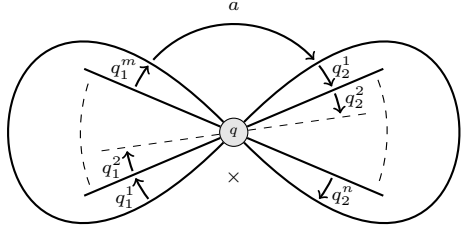
The interpretation of basis elements corresponding to arrows in fundamental cycles in case (3) in Theorem~\ref{HH^1} is given in Figure~\ref{fig:ExampleCuts}, and requires the definition of a fundamental cycle given in \ref{subsection::fundamental cycles}.

\vspace{-.25cm}
\begin{figure}[H]
	{\tiny
		\centering
		\begin{tikzpicture}[auto, thick, yscale=0.9]
		\clip (-1.5,-1) rectangle (2.2,1);
		\node[cg] (1) at (0,0){\,\,\,\,};
		\node () at ($(1)+(0:0.8)$){$a$};
		\draw[-] (1) to [out=45, in=-45, looseness=35] node {}(1);
		\draw[->] (1) +(33:.6cm) arc (33:-33:.6cm);
		\draw[-,dashed] (1) to node  {} ($(1)+(150:1.3)$);
		\draw[-,dashed] (1) to node  {} ($(1)+(210:1.3)$);
		\end{tikzpicture}
	}
    	\caption{This local configuration corresponds to Theorem~\ref{HH^1} (4) and only contributes a dimension to  $HH^1(A)$ in case ${\rm char } K=2$.}
	\label{fig:config 6}
\end{figure}
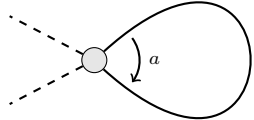
Before proving Theorem~\ref{HH^1},  we need the preliminary  results in \ref{fundamental cycles} and \ref{shortcuts deviations and loops} below. 

\subsubsection{ Fundamental cycles}
\label{fundamental cycles}\label{subsection::fundamental cycles}
Let $Q$ be a finite connected quiver and  $G_Q$ be the unoriented graph underlying $Q$ with $V$ vertices and $E$ edges. We now recall some standard notions from graph theory. 
A \emph{simple cycle} is a cycle in $G_Q$ such that each vertex appears only once. A \emph{spanning tree} is a subgraph $T$ of $G_Q$ which is a tree such that each vertex of $G_Q$ is in  $T$ and such that adding any edge 
$e \in G_Q \setminus T$  to $T$, then $T \cup {e}$ is not  a tree. Thus each edge $e \in G_Q \setminus T$ creates a cycle and these are the \emph{fundamental cycles} of $G_Q$. Note that the number of edges in the spanning tree $T$ is $V-1$ and so there are $E-V+1$ fundamental cycles in $G_Q$. By definition, $\chi(G_Q) = V-E$. So the number of fundamental cycles is $1 - \chi(G_Q)$.     

The \emph{heart} $H$ of $G_Q$ is the union of all simple cycles in $G_Q$. Let $D_i$, for $ 1 \leq i \leq r$, be the connected components in $H$ and $T_j$, for $1 \leq j \leq k$, the trees connecting the $D_i$. 

From now on we will identify $H$,  the $D_i$ and  the $T_j$ --which are subsets of $G_Q$--  with the respective subquivers of $Q$, that is, taking into account the orientation of the arrows.

	\begin{definition}{\rm
	(1) The {\em skeleton} $S$ of $Q$ is the minimal subquiver that connects $H$.\\
	(2)	The tree $f$ is a {\em branch} of $Q$ if it is a connected maximal subquiver and the arrows in $H$ or the skeleton $S$ do not appear in $f$. Denote by $F$ the union of the branches of $Q$.}
	\end{definition}	
Note that $H_1\cap S_1=\emptyset$ and $Q_1=H_1\sqcup S_1 \sqcup  F_1$, with $\sqcup$ the disjoint union. 

 
We  denote by   $\chi(\Gamma_A) = \vert (\Gamma_A)_0 \vert - \vert (\Gamma_A)_1 \vert$ the Euler characteristic of the graph underlying the ribbon graph $\Gamma_A$.

\begin{proposition}
Let $A= KQ/I$ be a  gentle algebra with ribbon graph $\Gamma_A$. Then $\chi(Q) = \chi(\Gamma_A)$ and moreover, $ \vert  \overline{\cM} \vert = 2 \vert Q_0 \vert - \vert Q_1 \vert$.
\end{proposition}
{\it Proof.} 
Set $\Gamma = \Gamma_A$. By definition of $\Gamma$, its edges are in one to one correspondence with the vertices in $Q$. 
 We have
$\chi(Q) = V_Q - E_Q   =  E_\Gamma  - \sum_{m \in \Gamma_0} (\val (m) -1) = E_\Gamma - \sloppy \sum_{m \in \Gamma_0} \val (m) + V_\Gamma$ where $\val(m)$ denotes the number of half-edges incident with the vertex $m$.   Since \sloppy $\sum_{m \in \Gamma_0} \val (m) = 2E_\Gamma$, we have $\chi(Q) = \chi(\Gamma)$. 
Therefore the number of vertices in $\Gamma$ is $2 \vert Q_0 \vert - \vert Q_1 \vert$ and thus $\vert  \overline{\cM} \vert = 2 \vert Q_0 \vert - \vert Q_1 \vert$.
\hfill $\Box$

\begin{corollary}\label{cor::same number fundamental cycles}
Let $A= KQ/I$ be a  gentle algebra with ribbon graph $\Gamma_A$. Then there is a bijection between the fundamental cycles in $G_Q$ and the fundamental cycles in $\Gamma_A$.
\end{corollary}

\begin{lemma}
Let $Q, D_i $ and $r$ be as above. Then $$1 - \chi(Q) = \sum_{i=1}^r (1 - \chi(D_i)).$$
\end{lemma}

\begin{proof}
Let $k$ be the number of trees connecting the $D_i$ and for each $i$, let $r_i$ be the number of connected components $D_j$ adjacent to the tree $T_i$. 
We have that $\chi(Q) = \sum_{i=1}^r  \chi(D_i) + \sum_{i=1}^k  \chi(T_i) - \sum_{i=1}^k  r_i$ and $ \sum_{i=1}^r (1 - \chi(D_i)) =  r - \sum_{i=1}^r  \chi(D_i)$. Since
 the $T_i$ are trees, we have that $\chi(T_i) = 1$. Thus we need to show that $r = 1 - k + \sum_{i=i}^k  r_i$. We show this by induction on $r$. For $r =1$, it is
  clear. Suppose that $r = 1 - k_r + \sum_{i=1}^{k_r} r_i$. Suppose that we have $r+1$ connected components in the heart. Then there are two cases, either 
  $k_{r+1} = k_r +1$ or $k_{r+1} = k_r $.
In the first case, we have  $1 - k_{r+1}  + \sum_{i=1}^{k_{r+1}} r_i = 1 - k_{r} -1   + \sum_{i=1}^{k_{r}} r_i + r_{k_{r+1}}$. By the induction hypothesis, this is $r -1 +  
r_{k_{r+1}}$. Since the new tree $T_{k_{r+1}}$ is adjacent to one existing connected component of the heart and the new one, $r_{k_{r+1}} = 2$. In the other 
case, we can choose a numbering of the connecting trees $T_j$ such that the new connected component is adjacent to $T_{k_r}$. Then 
$1 - k_{r+1}  + \sum_{i=1}^{k_{r+1}} r_i = 1 - k_{r}   + \sum_{i=1}^{k_{r} - 1}  r_i  + r'_{k_r}$ where $r'_{k_r} = r_{k_r} +1$. Thus we have $1 - k_{r}   + \sum_{i=1}^{k_{r}  }  r_i  - r_{k_r} + r'_{k_r}$ and   the result follows by induction. 
\end{proof}

\begin{proposition}\label{a||a}
	Each fundamental cycle in the graph $G_Q$ contributes one dimension to $HH^1(A)$.  Moreover, for each fundamental cycle, choosing an arrow $a$ such that the remaining arrows form a spanning tree of $Q$, the class of $a \vert \vert a$ belongs to a basis of $HH^1(A)$. 
\end{proposition}

Before proving the proposition we show how to identify a set of arrows in the fundamental cycles of the ribbon graph of $A$ such that the corresponding elements are linearly independent in $HH^1(A)$.  By Corollary~\ref{cor::same number fundamental cycles}, to each fundamental cycle in $G_Q$ corresponds a fundamental cycle in $\Gamma_A$. Furthermore, by \cite{S} the arrows in $G_Q$ are in bijection with the (non-marked) angles in $\Gamma_A$ --an angle, is given by two consecutive (in the cyclic ordering) half-edges incident to the same vertex.  Choosing an  arrow $a$ in each fundamental cycle in $G_Q$ corresponds to cutting $\Gamma_A$ at the (non-marked) angle corresponding to $a$. Conversely any cut at a non-marked angle corresponds to an arrow in $G_Q$ and a minimal set of cuts in $\Gamma_A$ such that the cut graph  is a tree gives rise to a subset of basis elements of $HH^1(A)$ of the form $(a,a)$ corresponding to a choice of one arrow in each fundamental cycle in $G_Q$. 
\vspace{-0.2cm}
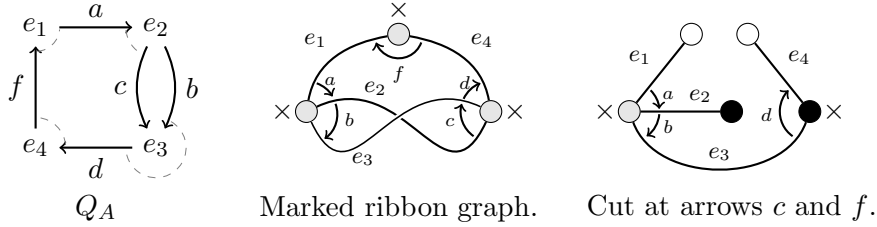
\begin{figure}[H]
	\centering
	\tikzstyle{help lines}+=[dashed]
	\begin{tikzpicture}[auto, thick, scale=0.8]+=[dashed]
	\node (1) at (0:0){$e_1$};
	\node (2) at ($(0)+(0:2)$) {$e_2$};
	\node (3) at ($(2)+(-90:2)$) {$e_3$};
	\node (4) at ($(0)+(-90:2)$) {$e_4$};
	\node ()  at ($(0)+(0:1)+(-90:3)$) {$Q_A$};
	
	\draw[->] (1) to node {$a$}(2);
	\draw[->] (2) to [out=-65, in=65] node {$b$}(3);
	\draw[style=help lines] (3) +(60:.5cm) arc (60:-175:.5cm);
	\draw[->] (2) to [out=-115, in=115] node [left]{$c$}(3);
	\draw[style=help lines] (2) +(185:.5cm) arc (185:245:.5cm);
	\draw[->] (3) to node {$d$}(4);
	\draw[style=help lines] (4) +(5:.5cm) arc (5:85:.5cm);
	\draw[->] (4) to node {$f$}(1);
	\draw[style=help lines] (1) +(-85:.4cm) arc (-85:-5:.4cm);
	\node[cg] (1b) at ($(2)+(-90:1.4)+(0:2.5)$){};
	\node[cg] (2b) at ($(1b)+(40:2)$){};
	\node[cg] (3b) at ($(2b)+(-40:2)$){};
	\node ()  at ($(2b)+(-90:2.9)$) {Marked ribbon graph.};
	\node () at ($(1b)-(0:.4)$){$\times$};
	\node () at ($(2b)+(90:.4)$){$\times$};
	\node () at ($(3b)+(0:.4)$){$\times$};
	{\footnotesize
		\draw[-] (1b) to [out=80, in=190] node {$e_1$}(2b);
		\draw[-] (2b) to [out=-10, in=100] node {$e_4$}(3b);
		\draw[-] (1b) to [out=25, in=245,looseness=1.5] node [above, pos=0.25] {$e_2$}(3b);
		\draw[draw=white,double=black] (1b) to [out=-65, in=155,looseness=1.5] node [below, pos=0.4] {$e_3$}(3b);
	}
	{\tiny
		\draw[->] (1b) +(70:.5cm) arc (70:30:.5cm) node [above, pos=0.8] {$a$};
		\draw[->] (1b) +(15:.5cm) arc (20:-55:.5cm) node [right, pos=0.4] {$b$};
		\draw[->] (2b) +(-20:.4cm) arc (-20:-160:.4cm) node [below, pos=0.5] {$f$};
		\draw[->] (3b) +(155:.5cm) arc (155:110:.5cm) node [above, pos=0.1] {$d$};
		\draw[->] (3b) +(235:.5cm) arc (235:165:.5cm) node [left, pos=0.5] {$c$};
	}
	\node[cg] (1c) at ($(3b)+(0:2.3)$){};
	\node[cw] (2c) at ($(1c)+(40:2)-(0:0.5)$){};
	\node[cg,black] (3c) at ($(1c)+(0:3)$){};
	\node[cw] (2cc) at ($(3c)+(140:2)+(0:0.5)$){};
	\node[cw,black] (3cc) at ($(1c)+(0:1.7)$){};
	\node ()  at ($(3cc)+(-90:1.6)$) {Cut at arrows $c$ and $f$.};
	\node () at ($(1c)-(0:.4)$){$\times$};
	\node () at ($(3c)+(0:.4)$){$\times$};
	
	{\footnotesize
		\draw[-] (1c) to  node {$e_1$}(2c);
		\draw[-] (2cc) to  node {$e_4$}(3c);
		\draw[-] (1c) to  node [above, pos=0.75] {$e_2$}(3cc);
		\draw[-] (1c) to [out=-70, in=250,looseness=1] node [above] {$e_3$}(3c);
	}
	{\tiny
		\draw[->] (1c) +(45:.5cm) arc (45:5:.5cm) node [right, pos=0.3] {$a$};
		\draw[->] (1c) +(-5:.5cm) arc (-5:-55:.5cm) node [right, pos=0.5] {$b$};
		\draw[->] (3c) +(235:.5cm) arc (235:135:.5cm) node [left, pos=0.5] {$d$};
	}
	\end{tikzpicture}
	\caption{Example of a minimal set of cuts given by the arrows $\{c,f\}$ corresponding to two basis elements $(c,c)$ and $(f,f)$ in $HH^1(A)$.}
	\label{fig:ExampleCuts}
\end{figure}

We now prove a series of lemmas which we will use in the proof of Proposition~\ref{a||a}.

	\begin{lemma} \label{0}
		Let $A = KQ/I$ be a gentle algebra and let $a$ be an arrow in $Q$,
		\begin{enumerate} 
			\item $(a,a)$ is a $1$-cocycle for each $a\in Q_1$.
			\item if $\val(s(a))=1$ or $\val(t(a))=1$, then  $(a,a)$ is a $1$-coboundary.
			\item if $a$ is a loop, then  $(a,a)$ is not a $1$-coboundary.
		\end{enumerate}
	\end{lemma}
	\begin{proof}
			(1) This follows immediately from \eqref{d0} and \eqref{d1} since $d^1(a,a) = \sum_{q\in R} (q, q^{(a, a)}) = \sum_{q\in R, a \in q} (q, q) = 0$ where the second sum runs  over those $q \in R$ such that $a$ is an arrow in $q$.  
			
			(2) Let $e$ be the vertex of $a$ with valency $1$, then $d^0(e,e)=\pm (a,a)$.
			
			(3) If $a$ is a loop, let $e=s(a)$ and $b,c \in Q_1$, we have by \eqref{d0}: 
			\begin{equation*}
			d^0(e,e)= \left\lbrace
			\begin{array}{ll}
			0& \text{if } \val (e)=2,\\
			\pm(b,b)& \text{if } \val (e)=3 \text{ and $b$ incides in } e,\\
			(b,b)-(c,c)& \text{if } \val (e)=4 \text{ and } s(b)=e=t(c),
			\end{array}\right.
			\end{equation*}		
			and
			\begin{equation*}
			d^0(e,a)=\left\lbrace
			\begin{array}{ll}
			0& \text{if } \val (e)=2 ,\\
			(b,ba)& \text{if } \val (e)=3 \text{ and } s(b)=e,\\
			-(b,ab)& \text{if } \val (e)=3 \text{ and } t(b)=e,\\
			(b,ba)-(c,ac)& \text{if } \val (e)=4 \text{ and } s(b)=e=t(c).
			\end{array}\right.
			\end{equation*}
			Hence $(a,a)$ does not belong to $\Im d^0$.
	\end{proof}
	
	\begin{lemma}\label{branch}
		Let $A = KQ/I$ be a gentle algebra and let $p_0$ be a branch of $Q$. If $a$ is an arrow in $p_0$, then the class of $(a,a)$ in $HH^1(A)$ is zero.
	\end{lemma}
	\begin{proof}
		Let  $P_{0}:=\{a\in Q_1| a\text{ is an arrow in  }p_0,\, \val _Q(s(a))=1 \text{ or } \val _Q(t(a))=1\}$. 
	Note that $P_0 \neq \emptyset$. By Lemma  \ref{0} we know that $(a,a)$ is a $1$-coboundary for each arrow $a$ in $P_{0}$. Let  $p_1$ be the subtree of $p_0$ without the arrows in $P_0$. If $p_1$ is trivial then the result follows. Now suppose that $p_1$ is not trivial and set \[P_{1}:=\{a\in Q_1| a\text{ is an arrow of } p_1,\val _{Q\setminus P_0}(s(a))=1 \text{ or } \val _{Q\setminus P_0}(t(a))=1\}.\] 
	Hence,  $P_1 \neq \emptyset$ and every $a\in P_{1}$ has at least one vertex which connects  with at least one arrow  $b\in P_{0}$.  If $e$ is this vertex then 
		\begin{equation*}
		d^0(e,e)=\left\lbrace
		\begin{array}{ll}
		\pm(a,a)\pm(b_1,b_1)& \text{ if } \val (e)=2 ,\\
		\pm(a,a)\pm(b_1,b_1)\mp(b_2,b_2)& \text{ if } \val (e)=3,\\
		\pm(a,a)\pm(b_1,b_1)\mp(b_2,b_2)\mp(b_3,b_3)& \text{ if } \val (e)=4,
		\end{array}\right.
		\end{equation*}
		where $b_1$, $b_2$ and $b_3$ are the arrows in $P_0$ which connect  to $e$. So the class $(a,a)$ can be written as the sum of zero classes in $HH^1(A)$. Since $Q$ is a finite quiver, repeating this process there is  some natural $t$ such that $P_t=\emptyset$.
	\end{proof}

	\begin{lemma}
		Let $A = KQ/I$ be a gentle algebra with skeleton $S$. Given an arrow $a$ in $S$, the class of $(a,a)$ in $HH^1(A)$ is zero.
	\end{lemma}
	\begin{proof}
		By  Lemma \ref{branch} we can assume $Q$ has no branches. Note that by the minimality property in the definition of the skeleton, there exists a connected component 
		$D$ in $H$ such that only one arrow of $S$ is incident to $D$. Let $x$ be this arrow. Denote by $D_0$ the vertices in $D$, since
		\[\sum_{e_i\in D_0}d^0(e_i,e_i)=\sum_{e_i\in D_0}\left( \sum_{b\in Q_1e_i}(b,b)-\sum_{b\in e_iQ_1}(b,b)\right)=(x,x),\]
		we have that $(x,x)$ is a $1$-coboundary. 
		
		Now, consider the subquiver $Q^{(1)}$ consisting of $Q\setminus D$ and the vertex of $(x,x)$ that belongs 
		to $D$. If $Q^{(1)}$ has a branch, then $x$ belongs to this branch and $\val _{Q^{(1)}}(s(x))=1$ or 
		$\val _{Q^{(1)}}(t(x))=1$. We can subsequently delete in this way all the classes $(b,b)$, 
		for each $b$ in the branch, up to the moment when we find an arrow with a vertex of valency greater than $2$; if this arrow is incident to another connected component of $H$, then we can start the process again, if not, we just delete this part of the branch and start the process again.
		If $Q^{(1)}$ does not have a branch, there exists a connected component $D^{(1)}$ in $H\setminus D$ such that only one arrow of $S$ is incident to this component, and we repeat the process, starting from $H\setminus D$. 
		
		Finally, since $Q$ is a finite quiver, $H$ is the union of $r$ connected components, and we conclude that 
		the class $(a,a)$ is zero in $HH^1(A)$ for each $a$ in $S$.
	\end{proof}

	\begin{lemma}\label{D}
		Let $A = KQ/I$ be a gentle algebra and suppose that $Q=D$. For each fundamental cycle in $Q$ there is a representative $(a,a)$ in the basis of  $HH^1(A)$.
	\end{lemma}

	\begin{proof}
		Let us assume without loss of generality that $D$ has no loops, let $T$ be a spanning tree of $D$ and let $U=D_1\setminus T_1$ be the set of arrows in $D$ that do not appear in $T$. We consider the matrix $M$ of the system $d^0(e_i,e_i)=0$ for each $e_i\in D_0$, where we place in the last columns the elements of $U$. There are as many rows as vertices and as many columns as arrows in $D$. But  $\rank M= |D_0|-1$ and the reduced matrix has the following shape
		\[ \left( \begin{array}{cccccccc}
		1 & 0 & 0 & ... & 0 & * &...& *\\
		0 & 1 & 0 & ... & 0 & * &...& *\\ 
		0 & 0 & 1 & ... & 0 & * &...& *\\ 
		&   &   &     &   &   &   & \\ 
		0 & 0 & 0 & ... & 1 & * &...& *\\	
		0 & 0 & 0 & ... & 0 & 0 &...& 0
		\end{array}\right).  \]
		Therefore, the arrows in $T$ are linear combinations of arrows of $U$ in $HH^1(A)$ and the class of $(a,a)$ in $HH^1(A)$ is not zero for each $a$ in $U$. Note that the pairs $(a,a)$ with $a\in U$ are linearly independent, because if not we would have a row with zero entries in the $U$ positions and with nonzero scalars in the other positions. 
		
		Now, if $D$ has loops, these loops are fundamental cycles and the result is obtained by Lemma \ref{0}.
	\end{proof}

The proof of Proposition~\ref{a||a} now easily follows from Lemmas~\ref{0} to \ref{D}.

\subsubsection{Shortcuts,  deviations and loops}\label{shortcuts deviations and loops}
In this subsection we show how to find the other elements  in the basis of $HH^1(A)$.

%
%
%
%
%

\begin{proposition}\label{lc} Let $A = KQ/I$ be a gentle algebra and let 
 $(a,p), (a',p') \in Q_1\vert \vert \cB$ be such that $q^{(a,p)}\neq 0$ for some $q\in R$. We have $q^{(a,p)}=q^{(a',p')}$ if and only if one of the following holds:
	\begin{itemize}
		\item $(a,p)=(a',p')$,
		\item $q=aa'$, $p=ar$ and $p'=ra'$, where $r$ is not a trivial cycle in $\cB$.
	\end{itemize}
\end{proposition}
\begin{proof}
	Suppose that  $0\neq q^{(a,p)}=q^{(a',p')}$, then the arrows $a$ and $a'$ appear in $q$. 
	If $a=a'$ and $q=aq'$, we have $q^{(a,p)}=pq'=p'q'$, hence $p=p'$; similarly we get the same if $q=q'a$. Therefore $(a,p)=(a',p')$.
	
	If $a\neq a'$, let us assume without loss of generality that $q=aa'$, then $q^{(a,p)}=pa'=ap'=q^{(a',p')}$, thus $a$ is the first arrow of $p$ and $a'$ is the last arrow of $p'$, i.e. $p=ar$ and $p'=r'a'$ for some cycles $r,r' \in \cB$. So $ara'=ar'a'$ and	$r=r'$.
	The converse immediately follows from the definitions.
\end{proof}

Recall the following from \cite[ Lemma 2.4.2.1]{Str}.

\begin{proposition}
	Let  $A$ be a monomial algebra and $d^1$ given by \eqref{d1}. The set $K(Q_1\vert \vert Q_0)\cap \Ker d^1$ is zero if and only if there exists a path $q\in R$ such that $q^{(a,s(a))}\neq0$ for any loop $a$. 
\end{proposition}
\begin{corollary}
		Let $A=KQ/I$ be a gentle algebra and $d^1$ given by \eqref{d1}. If ${\rm char }  K\neq 2$ then $K(Q_1\vert \vert Q_0)\cap \Ker d^1$ is zero.
\end{corollary}
\begin{proof}
	If $a$ is a loop then $(a^2,2a)$ is a summand of $d^1(a,s(a))$.
\end{proof}

\begin{corollary}\label{loop}
	Let $A=KQ/I$ be a gentle algebra with ${\rm char } K= 2$ and $a$ be a loop, then the class of $(a,s(a))$ is an element of the basis of $HH^1(A)$
\end{corollary}
\begin{proof}
	Note that $(a,s(a))$ is not a $1$-coboundary, and $d^1(a,s(a))=0$.
\end{proof}

	{\it Proof of Theorem~\ref{HH^1}.}
			(1) Let $a$ be a shortcut for $q$, then  $(a,q)$ is not a $1$-coboundary, this immediately follows from the  definition of $d^0$. Since $a$ is not in any relation in $I$, the element $(a,q)$ is a $1$-cocycle. Note that if $a$ appears in at least one relation, by Proposition \ref{lc} it is not possible for $(a,q)$ to be a $1$-cocycle.
			
			(2) Let $q=q_2aq_1$, then we have three cases:
			
(i) If $q_1,q_2 \notin Q_0$ then $(a,q)$ is not a $1$-coboundary. If $q$ is not a deviation for $a$ then there exists a path $r\in R$ such that $a$ is an arrow in $r$ and $(r,r^{(a,q)})\neq 0$. By Proposition \ref{lc} the element $(a,q)$ is not a $1$-cocycle. if $q$ is a deviation via $a$ then $(a,q)$ is a $1$-cocycle. 
				
(ii) If $q_1 \notin Q_0$  and $q_2\in Q_0$ then $q=aq_1$ and we have two cases. The first is when $t(a)$ is not a vertex in $q_1$ and the second when $t(a)$ is a vertex in $q_1$. In the first case there exists at most an arrow $b$  which is not the last arrow of $q_1$ and $t(b)=s(a)$. If there exists $b$ then $ab$ is in $R$ because $A$ is gentle.
\begin{figure}[H]
	\centering
	\tikzstyle{help lines}+=[dashed]
	\begin{tikzpicture}[auto, thick,yscale=0.9]+=[dashed]
	{\footnotesize
		\node (0) at (0:0) {$\bullet$};
		\node (1) at ($(0)+(0:2)$) {$\bullet$};
		\node (2) at ($(1)+(0:2)$) {$\bullet$};
		\draw[->,dashed] (1) to [out=130, in=50, looseness=15] node {$q_1$}(1);
		\draw[->] (0) to node {$b$} (1);
		\draw[->] (1) to node {$a$} (2);
		\draw[style=help lines] (1) +(200:.45cm) arc (200:350:.45cm);
		\draw[style=help lines] (1) +(55:.55cm) arc (55:130:.55cm);
	}
	\end{tikzpicture}
	\caption{Local configuration in $Q$ if there exists $b$.}
	\label{fig:case4-1}
\end{figure}
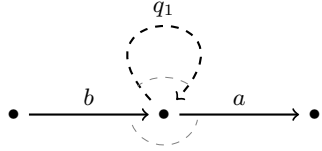
Therefore
		\begin{equation*}
			d^0(s(q_1),q_1)= \left\lbrace
			\begin{array}{ll}
				(a,aq_1) & \text{if } \val (s(a))=3,\\
				(a,aq_1)-(b,q_1 b)& \text{if } \val (s(a))=4,
			\end{array}\right.
		\end{equation*}
		and 
		\begin{equation*}
			d^1(a,aq_1)= \left\lbrace
			\begin{array}{ll}
					0& \text{if } \val (s(a))=3,\\
					(ab,aq_1b)& \text{if } \val (s(a))=4.
			\end{array}\right.
		\end{equation*}
	So, if $b$ exists, this  implies that $(a,aq_1)$ is not a $1$-cocycle. In addition $(a,aq_1)-(b,q_1b)$ is zero in $HH^1(A)$. On the other hand, if $b$ does not exist then $(a,aq_1)$ is zero in $HH^1(A)$.
	
	Similarly, if $t(a)$ is a vertex in $q_1$, then $q_1= r_2t(a)r_1$ with $r_1$ and $r_2$ paths in $\cB$ and  there exists at most an arrow $b$ such that $t(b)=s(a)$ and $ab$ is in $R$.
		\begin{figure}[H]
		\centering
		\tikzstyle{help lines}+=[dashed]
		{\footnotesize
			\begin{tikzpicture}[auto, thick,yscale=0.9]+=[dashed]
					\node (0) at (0:0) {$\bullet$};
					\node (1) at ($(0)+(0:2)$) {$\bullet$};
					\node (2) at ($(1)+(0:2.8)$) {$\bullet$};
					
					\draw[->] (0) to node {$b$} (1);
					\draw[->] (1) to [out=40, in=140, looseness=1]  node {$a$} (2);
					\draw[->,dashed] (2) to node [below]{$r_2$} (1);
					\draw[->,dashed] (1) to [out=-40, in=220, looseness=1]   node  [below] {$r_1$} (2);
					\draw[style=help lines] (1) +(50:.45cm) arc (50:170:.45cm);
					\draw[style=help lines] (1) +(330:.7cm) arc (330:360:.7cm);
					\draw[style=help lines] (2) +(145:.7cm) arc (145:180:.7cm);
			\end{tikzpicture}}
		\caption{Local configuration in $Q$ if there exists $a$.}
		\label{fig:case4-2}
		\end{figure}
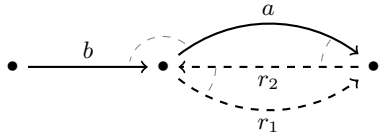
The differentials are as above and we obtain the same conclusion.

(iii) If $q_1 \in Q_0$  and $q_2\notin Q_0$ the proof is similar to  the previous case.
	
	Part (3) follows from  Proposition \ref{a||a} and part (4) follows from Corollary \ref{loop}.
	\hfill$\Box$

\bigskip

As a consequence of Theorem~\ref{HH^1} we immediately recover Theorem 1 in  \cite{Cibils-Saorin}.

\subsection{$HH_1(A)^*$ of a gentle algebra $A$}

We will interpret the dual $HH_1(A)^*$ of the first Hochschild homology group $HH_1(A)$ in terms of cyclic pairs. Before stating the main theorem of this section, let us fix some notation.

We denote by $\cB^*$ the dual basis of the vector space $DA$.

\begin{definition}{\rm 
		Two paths $p, q$ of $Q$ are called a {\em cyclic pair} if $s(p)=t(q)$ and $t(p)=s(q)$. If $X$ and $Y$ are sets of paths of $Q$, the {\em set $X\odot Y$ of cyclic pairs} is formed by couples $(p,q)$ in $X\times Y$ such that $p$ and $q$ are a cyclic pair. Also, $X\odot Y^*$ is the set formed by $(p,q^*) \in X\times Y^*$ with $(p,q)\in X\odot Y$ and denote by $K(X\odot Y^*)$ the vector space with basis the set $X\odot Y^*$.}
\end{definition}

\begin{theorem}\label{basisHH_1}  Let $A = KQ/I$ be a gentle algebra and $a, b \in Q_1$, $e \in Q_0$. 
	The classes of the following elements form a basis of $HH_1(A)^*$. 
	\begin{enumerate}
			\item $(a,b^*)-(b,a^*) \in K(Q_1 \odot \cB^*)$, where $ab,ba \in I$. 
		 \item $(a, e^*) \in K(Q_1 \odot \cB^*)$.
		\item If ${\rm char } K=2$ in addition to the above, we have the elements $(a,a^*) \in K(Q_1 \odot \cB^*)$.
	\end{enumerate}
\end{theorem}

Note that in cases (2) and (3) it follows from the fact that $A$ is finite dimensional and gentle that $a^2 \in I$. 

\begin{figure}[H]
	{\tiny
		\centering
		\begin{tikzpicture}[auto, thick, xscale=1.1, yscale=0.9]
		\clip (-1,-2.1) rectangle (11.1,1.2);
		\node[cg] (0) at (0:0){\,\,\,};
		\node () at ($(0)-(0:.5)$){$\times$};
		\node[cg] (1) at ($(0)+(0:2.5)$){\,\,\,};
		\node () at ($(1)+(0:.5)$){$\times$};
		\draw[-] (0) to [out=60, in=120, looseness=1.1] node   {} (1);
		\draw[-] (1) to [out=240, in=300, looseness=1.1]  node  {} (0);
		
		\draw[-,dashed] (0) to node  {} ($(0)+(135:1.3)$);
		\draw[-,dashed] (0) to node  {} ($(0)+(225:1.3)$);
		\draw[-,dashed] (1) to node  {} ($(1)+(45:1.3)$);
		\draw[-,dashed] (1) to node  {} ($(1)+(-45:1.3)$);
		
		\draw[->] ($(0)+(40:1)$) to [bend left=30] node [left]{$a$} ($(0)+(-40:1)$);
		\draw[->] ($(1)+(220:1)$) to [bend left=30]  node [right] {$b$} ($(1)+(140:1)$);
		{\normalsize \node () at ($(0)+(0:1.25)+(-90:1.7)$){(1.a)};}
		\node[cg] (1b) at ($(1)+(0:3.4)$){\,\,\,};
		\node () at ($(1b)-(0:.45)$){$\times$};
		\node () at ($(1b)+(90:0.8)$){$a$};
		\node () at ($(1b)+(-90:0.8)$){$b$};
		\draw[-] (1b) to [out=50, in=-50, looseness=30] node {}(1b);
		\draw[-] (1b) to [out=130, in=230, looseness=30] node {}(1b);
		\draw[->] (1b) +(125:.6cm) arc (125:55:.6cm);
		\draw[->] (1b) +(-55:.6cm) arc (-55:-125:.6cm);
		\draw[-,dashed] (1b) to node  {} ($(1b)+(150:1.1)$);
		\draw[-,dashed] (1b) to node  {} ($(1b)+(210:1.1)$);
		\draw[-,dashed] (1b) to node  {} ($(1b)-(150:1.1)$);
		\draw[-,dashed] (1b) to node  {} ($(1b)-(210:1.1)$);
		{\normalsize \node () at ($(1b)+(-90:1.7)$){(1.b)};}
		
		\node[cg] (2) at ($(1b)+(0:3.4)$){\,\,\,};
		\node () at ($(2)-(0:.4)$){$\times$};
		\node () at ($(2)+(0:0.8)$){$a$};
		\draw[-] (2) to [out=45, in=-45, looseness=29] node {}(2);
		\draw[->] (2) +(35:.6cm) arc (35:-35:.6cm);
		\draw[-,dashed] (2) to node  {} ($(2)+(135:1.3)$);
		\draw[-,dashed] (2) to node  {} ($(2)+(225:1.3)$);
		{\normalsize \node () at ($(2)+(0:0.25)+(-90:1.7)$){(2)-(3)};}
		\end{tikzpicture}
	}
	\caption{Local configurations giving rise to basis elements of $HH_1(A)^*$ as described in Theorem~\ref{basisHH_1} (1) and (2) (and (3)).}
	\label{fig:config 2}
\end{figure}
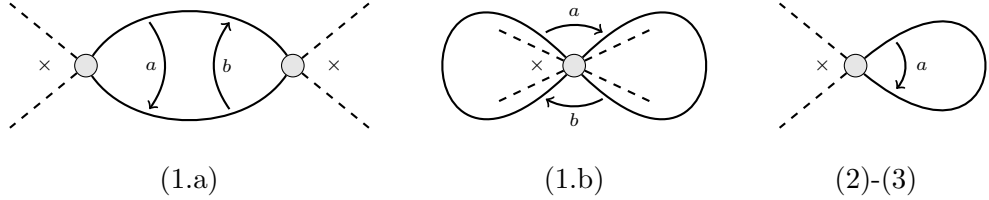

\begin{proposition}
	Let $A=KQ/I$ be a monomial algebra, the dual of the Hochschild homology of $A$ in low degrees is the cohomology of the cochain complex
	\[0\longrightarrow K(Q_0\odot \cB^*)\xrightarrow{d_0} K(Q_1\odot \cB^*) \xrightarrow{d_1} K(R\odot \cB^*)\rightarrow...\]
	where \[d_0(e,p^*)=\sum_{a\in Q_1e} \chi_{\cB^*}(ap^*)(a, ap^*)-\sum_{a\in eQ_1} \chi_{\cB^*}(p^* a)(a, p^* a),\]
	\[d_1(a,p^*)=\sum_{q=\theta a \nu} \chi_{\cB^*}(\theta p^* \nu) (q, \theta p^* \nu). \]
\end{proposition}
\begin{proof}
The proof is similar to the one of Theorem 4.2.2 (iii) in  \cite{Str}.
\end{proof}
	
	Note that $K(Q_0\odot \cB^*)$ is generated by $(e, e^*)$, for every $e\in Q_0$, and $(s(p), p^*)$ where $p$ is a cycle in $\cB$. The differential $d_0$ is as follows:
	\[ \begin{array}{lcll}
			d_0: & K(Q_0 \odot \cB^*)& \longrightarrow & K(Q_1\odot \cB^*)\\
			&(e, e^*) & \longmapsto & 0 \\
			& (s(p), p^*) & \longmapsto & \left\lbrace 
				\begin{array}{l}
					0 \text{ if } p\in Q_1,\\
					(p_1, (p_n ... p_2)^*)-(p_n , (p_{n-1}...p_1)^*)  \text{ if } p=p_n...p_1\in \cB\setminus Q_1.
				\end{array}
	 			\right. 
	\end{array}\]
	$K(Q_1\odot \cB^*)$ is generated by the elements
	\begin{itemize}
		\item $(a,s(a)^*)$ and $(a,a^*)$ where $a$ is a loop,
		\item $(a, q^*)$  where $aq$ is a cycle in $Q$, $a\in Q_1$, $q\in \cB\setminus Q_0$ and  $aq, qa \in I$,
		\item $(a_1,  (a_n...a_2)^*)$ and $(a_n,  (a_{n-1}...a_1)^*)$ with $a_n...a_1$ a cycle in $\cB$ and $n\geq2$.
	\end{itemize}
	Given $q=q_m...q_1$ in $\cB$ with $q_1,...,q_m $ arrows, the differential $d_1$ is given by:
	\[\begin{array}{lcll}
		d_1: & K(Q_1\odot \cB^*) & \longrightarrow & K(R\odot \cB^*) \\
		&(a, s(a)^*) & \longmapsto & 0 \text{ if } a \text{ is a loop}\\
		& (a, a^*) & \longmapsto & 2(a^2, s(a)^*) \text{ if } a \text{ is a loop}\\
		& (a,q^*) &\longmapsto &\left\lbrace 
	\begin{array}{l}
		(aq, t(a)^*) +(qa, s(a)^*) \text{ if } q\in Q_1 \text{ and } aq,qa\in I,\\
		(q_1 a , (q_m...q_2 )^*) + (aq_m , (q_{m-1}...q_1 )^*)  \text{ if }  q\notin Q_ 1, aq,qa\in I.
	\end{array}\right. \\
		& (a_1, (a_n...a_2)^*)& \longmapsto & (a_1a_n, (a_{n-1}...a_2)^*)\\
		& (a_n, (a_{n-1}...a_1)^*)& \longmapsto &(a_1a_n, (a_{n-1}...a_2)^*).
	\end{array}\]
	
The proof of Theorem  \ref{basisHH_1} is a direct consequence of the above formulas.

\subsection{$Alt_{A} (DA)$ of a gentle algebra $A$}\label{Alt}
	We use the notation and results of \cite{CRS} to find the following basis of $Alt_{A} (DA)$.

\begin{theorem}\label{basisAlt}

 Let $A = KQ/I$ be a gentle algebra. 
 
 (1) Let $p,q \in \cB$ such that $p \neq q$, $pq,qp \in I$, $s(p) = t(q)$, $s(q) = t(p)$, and $\val (s(p)) = \val (s(q)) =2$. For such pairs $\{p, q\} \in \cB$ define a function 
	$$\phi_{\{p,q\}}: DA \to A$$ given by $p^* \mapsto q$, $q^* \mapsto -p$ and zero otherwise. 
	
	 Suppose ${\rm char} K \neq 2$. Then the set of  functions $\phi_{\{p,q\}}$ as above is a basis of $Alt_A(DA)$.
%

	(2) Suppose ${\rm char} K =2$.  Let  $p  \in \cB \setminus Q_0$ be such that  $s(p)=t(p)=e$, and $\val(e)=2$. 
			For such $p$ define a function   $$\psi_{\{e,p\}}: DA \to A$$ given by $\psi_{\{e,p\}} (q^*) = r$ if $p=qr$ or  $p=rq$  and zero otherwise, and   $$\psi_{\{p,p\}}: DA \to A$$ given by $\psi_{\{p,p\}} (p^*) = p$. Then the set consisting of the functions 
	$\phi_{\{p,q\}}$ with $p,q$ as in (1), $\psi_{\{e,p\}}$ and $\psi_{\{p,p\}}$ with $p$ as in (2) is a basis of   $Alt_A(DA)$.
\end{theorem}

Note that for the definition of $\phi_{\{p,q\}}$  the order in the pair does in fact affect the signs. However, we disregard this since  the parametrisation of the basis and the Lie structure do not depend on this choice.

In terms of the ribbon graph of $A$, the basis elements of $Alt_{A} (DA)$ correspond to the following local configurations in the ribbon graph $\Gamma_A$. 

\begin{figure}[H]
	{\tiny
		\centering
		\begin{tikzpicture}[auto, thick, xscale=0.9,yscale=0.85]
		\clip (-2,-2.3) rectangle (9,1.6);
		\node[cg] (0) at (0:0){$p$};
		\node[cg] (1) at ($(0)+(0:2.5)$){$q$};
		\node () at ($(0)+(0:.6)$){$\times$};
		\node () at ($(1)-(0:.6)$){$\times$};
		\draw[-] (0) to [out=60, in=120, looseness=1.1] node   {} (1);
		\draw[-] (1) to [out=240, in=300, looseness=1.1]  node  {} (0);
		\draw[-] (0) to node {} ($(0)+(120:1.55)$);
		\draw[-] (0) to node {} ($(0)+(150:1.55)$);
		\draw[-] (0) to node {} ($(0)+(210:1.55)$);
		\draw[-] (0) to node {} ($(0)+(240:1.55)$);
		
		\node () at ($(0)+(280:1.1)$){$p_1$};
		\draw[->] (0)+(305:.8cm) arc (305:245:.8cm);
		\node () at ($(0)+(225:1.1)$){$p_2$};
		\draw[->] (0)+(235:.8cm) arc (235:215:.8cm);
		\draw[-,dashed] (0)+(205:.8cm) arc (205:155:.8cm);
		\node () at ($(0)+(135:1.2)$){$p_{n-1}$};
		\draw[->] (0)+(145:.8cm) arc (145:125:.8cm);
		\node () at ($(0)+(80:1.1)$){$p_{n}$};
		\draw[->] (0)+(115:.8cm) arc (115:55:.8cm);
		\draw[-] (1) to node {} ($(1)+(60:1.55)$);
		\draw[-] (1) to node {} ($(1)+(30:1.55)$);
		\draw[-] (1) to node {} ($(1)+(-30:1.55)$);
		\draw[-] (1) to node {} ($(1)+(-60:1.55)$);
		
		\node () at ($(1)+(95:1.1)$){$q_1$};
		\draw[->] (1)+(125:.8cm) arc (125:65:.8cm);
		\node () at ($(1)+(45:1.1)$){$q_2$};
		\draw[->] (1)+(55:.8cm) arc (55:35:.8cm);
		\draw[-,dashed] (1)+(25:.8cm) arc (25:-25:.8cm);
		\node () at ($(1)+(-42:1.25)$){$q_{m-1}$};
		\draw[->] (1)+(-35:.8cm) arc (-35:-55:.8cm);
		\node () at ($(1)+(-95:1.1)$){$q_{m}$};
		\draw[->] (1)+(-65:.8cm) arc (-65:-125:.8cm) ;
		{\normalsize \node () at ($(0)+(0:1.25)+(-90:1.9)$){(1)};}
		\node[cg] (2) at ($(1)+(0:4.5)$){$p$};
		\node () at ($(2)+(0:.6)$){$\times$};
		\draw[-] (2) to [out=40, in=-40, looseness=25] node {}(2);
		\draw[-] (2) to node {} ($(2)+(120:1.7)$);
		\draw[-] (2) to node {} ($(2)+(150:1.55)$);
		\draw[-] (2) to node {} ($(2)+(210:1.55)$);
		\draw[-] (2) to node {} ($(2)+(240:1.7)$);
		
		\node () at ($(2)+(280:1.1)$){$p_1$};
		\draw[->] (2)+(315:.8cm) arc (315:245:.8cm);
		\node () at ($(2)+(225:1.1)$){$p_2$};
		\draw[->] (2)+(235:.8cm) arc (235:215:.8cm);
		\draw[-,dashed] (2)+(205:.8cm) arc (205:155:.8cm);
		\node () at ($(2)+(137:1.15)$){$p_{n-1}$};
		\draw[->] (2)+(145:.8cm) arc (145:125:.8cm);
		\node () at ($(2)+(80:1.1)$){$p_{n}$};
		\draw[->] (2)+(115:.8cm) arc (115:45:.8cm);
		{\normalsize \node () at ($(2)+(-90:1.9)$){(2)};}
		\end{tikzpicture}
	}
	\caption{Local configurations of $\Gamma_A$ contributing one basis element to $Alt_A(DA)$ where (2) only occurs in ${\rm char} K =2$ and where $p = p_n \ldots p_1$ and $q = q_m \ldots q_1$. }
	\label{fig:config 3}
\end{figure}
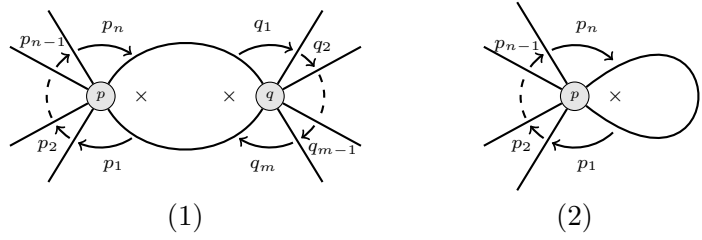

	To prove the theorem recall that 
	\[Alt_A(DA)=\{\psi\in Hom_{A-A}(DA,A) \mid f\psi(g)+\psi(f)g=0 \text{ for any} f,g \in DA\}.\]
	By adjunction we have
	\[Hom_{A-A}(DA,A)\cong Hom_{K}(DA\otimes_{A-A} DA, K).\]
	It is clear that  $\{p^*\otimes_{A-A}q^*\mid (p, q)\in \cB\odot \cB \}$ generates  $DA\otimes_{A-A} DA$.
	
	\begin{definition}
		A cyclic pair $(p,q) \in \cB\odot \cB$  is said to be neat if the following conditions hold:
		\begin{enumerate}[i)]
			\item if $q a \in \cB$ (resp. $aq \in \cB$)  for some $a\in Q_1$ then $a$ is the last (resp. first) arrow in $p$.
			\item if $p b \in \cB$ (resp. $bp \in \cB$)  for some $b\in Q_1$ then $b$ is the last (resp. first) arrow in $q$.
		\end{enumerate}
	\end{definition}
	
	Define a equivalence relation  $\sim$ on $\cB\odot \cB$  generated  by 
	\begin{equation*}
	\begin{array}{ccl}
		(ap,q) & \sim & (p,q a),\\
		(p b,q) & \sim & (p,bq),\\
	\end{array}
	\end{equation*}
	where $ a,b \in Q_1$ and $p,q \in \cB$. An equivalence class under this relation will be called neat when all its elements are neat. Denote by $\mathcal{N}$  the set of neat equivalence classes.
	
	In \cite{CRS} the authors show that there is a bijective map between $\mathcal{N}$ and a set of generators of the vector space $DA\otimes_{A-A} DA$. They also show that  the set $\{\psi_N \}_{N\in \mathcal{N}}$ is a basis of the vector space  $Hom_{A-A}(DA,A)$), where $\psi_N:DA\longrightarrow A$ is given by $\psi_N(p^*)=\sum_{(p,q)\in N}q $ for any $p^* \in \cB^*$. Furthermore, 
	if $\psi=\sum_{N\in\mathcal{N}}\lambda_N \psi_N$ is an element of $Hom_{A-A}(DA,A)$,  then $\psi$ is in $Alt_A(DA)$ if and only if  $\lambda_N + \lambda_{\upsilon(N)}=0$  for all  $N\in \mathcal{N}$, where $\upsilon$ is the flip of pairs.

	{\it Proof of Theorem~\ref{basisAlt}.} We begin by  determining the cycles in $Q$  that give rise to neat equivalence classes. Since $A$ is gentle we have the following local configurations  in $Q$, where $p, q \in \cB\setminus Q_0$.
	
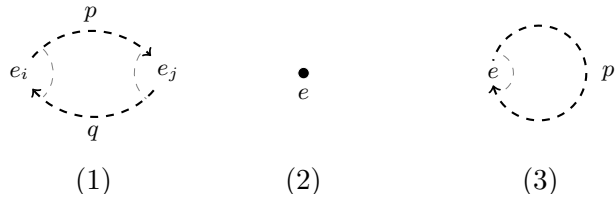
\vspace{-0.2cm}
		\begin{figure}[H]
		\centering
		\tikzstyle{help lines}+=[dashed]
		\begin{tikzpicture}[auto, thick, scale=0.9, every node/.style={scale=1}]+=[dashed]
		{\footnotesize
		\clip (-1.5,-1.78) rectangle (7.9,1.2);
		\node (0) at (0:0){};
		\node (1) at ($(0)-(0:1.1)$){$e_i$};
		\node (1a) at ($(0)+(0:1.1)$){$e_j$};
		\draw[->,dashed] (1) to [out=50, in=135, looseness=1] node  {$p$} (1a);
		\draw[->,dashed] (1a) to [out=230, in=310, looseness=1]  node  {$q$} (1);
		\draw[style=help lines] (1) +(45:.5cm) arc (45:-45:.5cm);
		\draw[style=help lines] (1a) +(145:.5cm) arc (145:230:.5cm);
		\node[bl] (e) at ($(0)+(0:3.1)$){};
		\node (2) at ($(e)+(-90:0.3)$) {$e$};
		\node (3) at ($(e)+(0:2.8)$){$e$};
		\node () at ($(3)+(0:1.7)$){$p$};
		\draw[<-,dashed] (3)+(-90:0.18) arc (-165:165:.7cm);
		\draw[style=help lines] (3) +(60:.3cm) arc (60:-65:.3cm);
		}
		\node () at ($(0)+(-90:1.6)$) {(1)};
		\node () at ($(e)+(-90:1.6)$) {(2)};
		\node () at ($(3)+(0:.7)+(-90:1.6)$) {(3)};
		
		\end{tikzpicture}
	\end{figure}
	Note that if there are other arrows adjacent to the vertices $e$, $e_i$ and $e_j$, then the corresponding cyclic pair is not neat. 
	
	(1) The classes $\overline{(p,q)}$ and $\overline{(q,p)}$  consist of only one element each and $\upsilon(p,q)=(q,p)$, thus  $\psi_{\overline{(p,q)}}-\psi_{\overline{(q,p)}} \in Alt_A(DA)$.
	
	(2) The class $\overline{(e, e)}=\{(e, e)\}$ is symmetric with regard to the involution $\upsilon$, thus  $\psi_{\overline{(e, e)}}\in Alt_A(DA)$ if and only if ${\rm char }  K=2$.
		
	(3) Let $p=p_n...p_1$ be a cycle in $\cB$ and  $p_i\in Q_1$ for any $i\in\{1,...,n\}$, the classes $\overline{(e,p)}=\{(e,p), (p_j...p_1,p_n...p_{j+1}),  (p, e),  (p_n...p_{j+1},p_j...p_1); j\in\{1,...,n-1\} \}$ and $\overline{(p,p)}=\{(p,p) \}$ are symmetric  and neat, then $\psi_{\overline{(p,p)}}, \psi_{\overline{(e,p)}}\in Alt_A(DA)$ if and only if ${\rm char }  K=2$.
\hfil $\Box$
		
\begin{Ex}
	Consider the algebra $A$ given by the following quiver with ribbon graph given by the figure on the right, and $I=\left\langle ca,ac,b^2\right\rangle$.
	
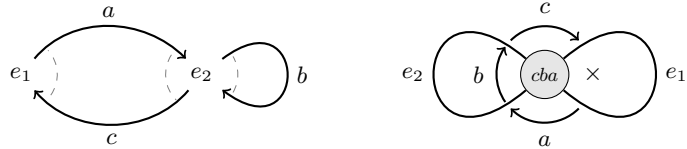
\vspace{-0.4cm}
	\begin{figure}[H]
		\centering
		\tikzstyle{help lines}+=[dashed]
		\begin{tikzpicture}[auto, thick,scale=.8]
			{\footnotesize
			\clip (-0.5,-1.3) rectangle (11.1,1.4);
			\node (1) at (0:0) {$e_1$};
			\node (2) at ($(1)+(0:3)$) {$e_2$};			
			\draw[style=help lines] (2) +(-30:.6cm) arc (-30:30:.6cm);
			\draw[style=help lines] (2) +(147:.6cm) arc (147:225:.6cm);
			\draw[style=help lines] (1) +(-35:.6cm) arc (-35:42:.6cm);
			\draw[->] (1) to [out=50, in=135, looseness=1]  node   {$a$} (2);
			\draw[->] (2) to [out=230, in=310, looseness=1]  node  {$c$} (1);
			\draw[->] (2) to [out=37, in=-37, looseness=9] node {$b$}(2);
			\node[cg] (1b) at ($(2)+(0:5.7)$) {$cba$};
			\draw[-] (1b) to [out=40, in=-40, looseness=13] node {$e_1$}(1b);
			\draw[-] (1b) to [out=-140, in=-220, looseness=13] node {$e_2$}(1b);
			\draw[->] (1b)+(315:.8cm) arc (315:225:.8cm);
			\node () at ($(1b)+(-90:1.1)$){$a$};
			\draw[->] (1b)+(215:.8cm) arc (215:145:.8cm);
			\node () at ($(1b)+(180:1.1)$){$b$};
			\draw[->] (1b)+(135:.8cm) arc (135:45:.8cm);
			\node () at ($(1b)+(90:1.1)$){$c$};
			\node () at ($(1b)+(0:.8)$){$\times$};}
		\end{tikzpicture}
	\end{figure}
\vspace{-0.4cm}
	The  cycle $p=cba$ is as in the case $(2)$ and the class $\overline{(e_1, p)}$ is neat. Moreover, 
	 $Alt_A(DA)=K \left( \psi_{\{p,p\}}\right)  \oplus K \left(  \psi_{\{e,p\}}\right)$ if ${\rm char } K=2$, otherwise  $Alt_{A}(DA)=0$.
\end{Ex}

\subsection{Hochschild cohomology of  Brauer graph algebras as  trivial extensions of  gentle algebras}

Let $\Lambda$ be a Brauer graph algebra with multiplicities all equal to 1. Using the fact that $\Lambda$ is a trivial extension of a gentle algebra $A$ and the vector space isomorphism $HH^1(TA)\cong Z(A)\oplus HH_1(A)^*\oplus HH^1(A)\oplus Alt_A(DA)$, we summarise the results of the previous sections and obtain the following basis of  $HH^1(\Lambda)$.

\begin{theorem} Let $\Lambda$ be a Brauer graph algebra with multiplicity one and let 
$A = KQ/I$ be a gentle algebra with basis of paths $\cB$ and such that $\Lambda \cong TA$. Suppose $p,q \in \cB$.
If ${\rm char } K\neq2$, then the set  consisting of the following elements determines a basis of 
$ HH^1(\Lambda) $:

	\begin{enumerate}
		\item $\sum_{e\in Q_0} (e,e)$,
		\item $(s(p),p) \in Q_0 \vert\vert \cB$  with $p$ a non trivial cycle and $\val (s(p))= 2$,
		\item  $(a,q) \in Q_1 \vert\vert \cB$ where $a$ is a shortcut for $q$,
		\item $(a,q)  \in Q_1 \vert\vert \cB$ with $q$ a deviation via $a$,
		\item $(a,a)  \in Q_1 \vert\vert \cB$ where $a$ is an arrow and we choose exactly one arrow in each fundamental cycle in $Q$,
		\item $(a,b^*)-(b,a^*) \in Q_1 \odot \cB^*$  with $ab, ba \in I$, 
		\item $(a, e^*)  \in Q_1 \odot \cB^*$  that is $a$ is a loop and $s(a)= e$,
		\item  $\phi_{\{p,q\}} \in Hom_{A-A}(DA,A)$  given   by $p^* \mapsto q$, $q^* \mapsto -p$  if $p\neq q$, $s(p) = t(q)$, 
		$s(q) = t(p)$,  $pq,qp \in I,$ and  $\val  (s(p)) = \val  (s(q)) =2$  
		 and given by the zero otherwise.
	\end{enumerate}
	If ${\rm char }  K=2$, in addition to the above we have  the following basis elements:
	\begin{enumerate}
		\item[(9)] $(a,s(a)) \in Q_1 \vert\vert \cB$ with $a$ a loop.
		\item[(10)] $(a,a^*) \in Q_1 \odot \cB^*$ with $a$ a loop, 
		\item[(11)]  $\psi_{\{s(p),p\}}, \psi_{\{p,p\}}  \in Hom_{A-A}(DA,A)$  where $p$ is a cycle with $\val (s(p))=2$ and
		\begin{enumerate}[i)]
				\item  $\psi_{\{s(p),p\}} (q^*) = r$ if $p=qr$ or  $p=rq$ and zero otherwise,
				\item  $\psi_{\{p,p\}} (p^*) = p$ and zero otherwise.
		\end{enumerate}
	\end{enumerate}
Furthermore,  let $\Gamma_\Lambda$ be the Brauer graph of $\Lambda$ given by the (unmarked) ribbon graph of $A$. Then  the local configurations in Figure~\ref{fig:config 1} (2.b), and Figures~\ref{fig:config 4}, \ref{fig:config 5}, \ref{fig:config 6}, \ref{fig:ExampleCuts}, \ref{fig:case4-1}, \ref{fig:case4-2}, \ref{fig:config 2} and \ref{fig:config 3} give rise to  a basis of  $HH^1(\Lambda)$.
\end{theorem}

 It follows from \cite{S} that given a Brauer graph algebra $\Lambda$ with multiplicity one, there exists a gentle algebra $B$ such that $\Lambda = TB$ and that the gentle algebra $B$ is not unique. In the next corollary we prove that there always exists a choice of a gentle algebra $B$ such that such that $Alt_B(DB) = 0$. 

	\begin{corollary}
		Let $A=KQ/I$ be a gentle algebra. There exists a gentle algebra $B$ such that $TA\cong TB$ and $Alt_B(DB)=0$.
	\end{corollary}
	\begin{proof}
		Consider the ribbon graph $\Gamma_A$ and remove the marking, this gives the Brauer graph of $TA$. Then one can always choose an admissible cut such that a local configuration (including the marking) as in  Figure \ref{fig:config 3}  does not appear. The marked ribbon graph corresponding to this cut determines a gentle algebra $B$ with $Alt_B(DB)=0$ and $TA\cong TB$.
	\end{proof}

 However, given a Brauer graph algebra $\Lambda$ with multiplicity one, it is not always possible to find a gentle algebra $A$ such that $\Lambda \cong TA$ and $HH_1(A) = 0$.

	\begin{Ex}
		Let $\Lambda$ be the Brauer graph algebra with the following  Brauer graph $G_\Lambda$ and quiver $Q_{\Lambda}$ where $k\geq 3$.
		
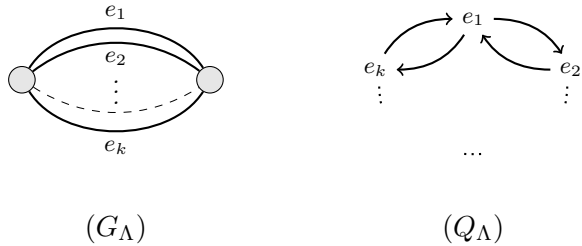
\vspace{-0.2cm}
		\begin{figure}[H]
			\centering
			\tikzstyle{help lines}+=[dashed]
			\begin{tikzpicture}[auto, thick,yscale=0.9]+=[dashed]
			{\footnotesize
			\node[cg](1) at (0:0){$\,\,\,$};
			\node[cg](1a) at ($(1)+(0:2.5)$){$\,\,\,$};
			\draw[-] (1) to [out=60, in=120, looseness=1] node   {$e_1$} (1a);
			\draw[-] (1) to [out=40, in=140, looseness=1] node [below]  {$e_2$} (1a);
			\draw[style={dashed},line width=0.1mm] (1) to [out=-35, in=215, looseness=1]  node [above] {$\vdots$} (1a);
			\draw[-] (1) to [out=300, in=240, looseness=1]  node [below] {$e_k$} (1a);
			
			\node (e1) at ($(1a)+(0:3.5)+(90:0.9)$){$e_1$};
			\node (e2) at ($(e1)+(-30:1.5)$){$e_2$};
			\node (ek) at ($(e1)-(30:1.5)$){$e_{k}$};
			
			\draw[->] (e1) to [out=0, in=120, looseness=1] node {} (e2);
			\draw[<-] (e1) to [out=180, in=60, looseness=1] node {} (ek);
			\draw[->] (e2) to [out=180, in=-60, looseness=1] node {} (e1);
			\draw[<-] (ek) to [out=0, in=240, looseness=1] node {} (e1);
			
			\draw[dotted] (e1) +(-87:2cm) arc (-87:-94:2cm);
			\draw[dotted] ($(e1)-(90:1)$) +(0:1.25cm) arc (0:-14:1.25cm);
			\draw[dotted] ($(e1)-(90:1)$) +(180:1.25cm) arc (180:194:1.25cm);
			}
		
			\node (x) at ($(0)+(-90: 2.2)+(0:1.25)$){$(G_\Lambda)$};
			\node () at ($(x)+(0: 3.5)+(0:1.25)$){$(Q_{\Lambda})$};
			\end{tikzpicture}
		\end{figure}
		\vspace{-0.4cm}
		Then for  any admissible cut  with cut algebra $A$ which by \cite{S} is gentle with $\Lambda \cong TA$, we have  $HH_1(A)\neq 0$. 
	\end{Ex}

\subsection{Two particular cases}\label{tpc}

 Suppose first that $A=K$. In this case the quiver $Q_A$ is just one vertex with no arrows, and so we have $A=KQ_A$. The centre is $K$, the Hochschild homology and cohomology are both zero in positive degrees and $Alt_A(DA)$ is also zero when ${\rm char }  K\neq 2$. 
If ${\rm char }  K=2$ then it is the $K$-vector space with one generator  $\psi_{\{e,e\}}$ where $e = 1_K$ corresponding to the map 
 $\psi_{\{e,e\}}:DA\rightarrow A$ such 
that $\psi_{\{e,e\}}(e^*)=e$.  Thus 
\[HH^1(TA)=\left\lbrace \begin{array}{ll}
	 K\oplus K\left( \psi_{\{e,e\}}\right)  & \text{if } {\rm char }  K= 2,\\
	K & \text{if } {\rm char }  K\neq 2.
\end{array}
\right. \]
	We will illustrate in this case the quiver $Q_A$, the ribbon graph $\Gamma_A$ and the quiver $Q_{TA}$:
	
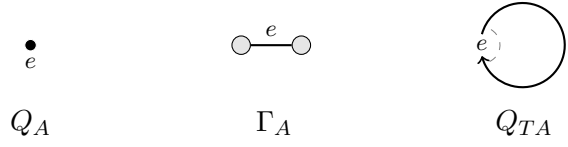
\begin{figure}[H]
		\centering
		\tikzstyle{help lines}+=[dashed]
		\begin{tikzpicture}[auto, thick,scale=0.8]+=[dashed]
		{\footnotesize
		\clip (-0.5,-1.5) rectangle (9,0.8);
		\node[bl] (0) at (0:0){};
		\node (1) at ($(0)+(-90:0.3)$) {$e$};
		\node[cg] (2) at ($(0)+(0:3.5)$){};
		\node[cg] (2b) at ($(2)+(0:1)$){};
		\draw[-] (2) to node {$e$} (2b);
		\node (3) at ($(2b)+(0:3)$){$e$};
		\node () at ($(3)+(0:1.7)$){$\alpha$};
		\draw[<-] (3)+(-90:0.18) arc (-165:165:.7cm);
		\draw[style=help lines] (3) +(60:.3cm) arc (60:-65:.3cm);
		}
		\node () at ($(0)+(-90:1.3)$) {$Q_A$};
		\node () at ($(2)+(0:.55)+(-90:1.3)$) {$\Gamma_A$};
		\node () at ($(3)+(0:.7)+(-90:1.3)$) {$Q_{TA}$};	
		\end{tikzpicture}
	\end{figure}
Now, suppose that $A=KQ/I$ where $Q$ has one vertex $e$ and one loop $x$, and $I=\left\langle x^2\right\rangle $. The centre of $A$ is $A$ while $HH^1(A)$ and $HH_1(A)$ depend on ${\rm char }  K$ (see \cite{GJRS, HT}). More precisely
	\[HH_1(A)\cong \frac{K[x]}{\left\langle x^2, 2x\right\rangle}= \left\lbrace 
	\begin{array}{ll}
	\frac{K[x]}{\left\langle x\right\rangle} \cong K& \text{ if } {\rm char }  K\neq 2, \\
	K\oplus Kx & \text{ if } {\rm char }  K=2,
	\end{array} \right. 
	\]
	and 
	\[HH^1(A)\cong Ann\left( (x^2) '\right)=Ann\left( 2x\right) = \left\lbrace 
	\begin{array}{ll}
	Kx& \text{ if } {\rm char }  K\neq 2, \\
	K\oplus Kx & \text{ if } {\rm char }  K=2,
	\end{array} \right. 
	\]
	the description of $Alt_A(DA)$ is, in this case,
	\[Alt_A(DA)= \left\lbrace 
	\begin{array}{ll}
	 K\left( \psi_{\{x,x\}}\right) \oplus K\left( \psi_{\{e,x\}}\right) & \text{ if } {\rm char }  K= 2, \\
	0 & \text{ if } {\rm char }  K\neq 2.
	\end{array} \right.
	\]
	The corresponding quivers and ribbon graph are
	
	
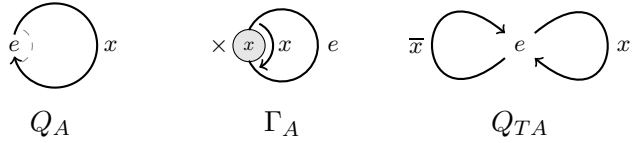
\begin{figure}[H]
		\centering
		\tikzstyle{help lines}+=[dashed]
		\begin{tikzpicture}[auto, thick, scale=0.8]+=[dashed]
		\clip (-0.5,-1.8) rectangle (10.5,0.9);
		{\footnotesize 
		\node (0) at (0:0){$e$};	
		\node (1) at ($(0)+(180:0.3)$) {};
		\node () at ($(1)+(0:1.9)$){$x$};
		\draw[<-] (0)+(-90:0.18) arc (-165:165:.7cm);
		\draw[style=help lines] (0) +(60:.3cm) arc (60:-65:.3cm);}
		\node () at ($(0)+(0:.6)+(-90:1.3)$) {$Q_A$};
		{\footnotesize  
		\node[cg] (2) at ($(0)+(0:3.9)$){$x$};
		\draw[-] (2)+(-90:.25) arc (-155:155:.6cm);
		\node () at ($(2)+(0:1.4)$) {$e$};
		\draw[->] (2)+(65:.4cm) arc (65:-65:.4cm);
		\node () at ($(2)+(0:0.6)$){$x$};
		\node () at ($(2)+(180:0.5)$){$\times$};
		}
		\node () at ($(2)+(0:.55)+(-90:1.3)$) {$\Gamma_{A}$};
		{\footnotesize  
		\node (3) at ($(2)+(0:4.5)$){$e$};
		\draw[->] (3) to [out=40, in=-40, looseness=13] node {$x$}(3);
		\draw[->] (3) to [out=-140, in=-220, looseness=13] node {$\overline{x}$}(3);}
		\node () at ($(3)+(-90:1.3)$) {$Q_{TA}$};
		\end{tikzpicture}
	\end{figure}

\section{The Lie algebra structure of $HH^1(A)$ and $HH^1(TA)$}

     We begin this section by recalling detailed results on how the first Hochschild (co)homology spaces of a monomial algebra embed into the first Hochschild cohomology  space of its trivial extension. The detailed description of the embeddings is necessary for the computation of the Gerstenhaber brackets. So we also recall some known results on the Gerstenhaber brackets and a well-known comparison map before explicitly determining the Lie algebra structure of the first Hochschild cohomology  group of gentle algebras and Brauer graph algebras.


		Let $A$ be a finite dimensional monomial algebra. In \cite[Th\'eor\`eme 4.1.0.3]{Str} it is shown that the  following $K$-linear map is an isomorphism
		\begin{equation*}
		\begin{array}{rcll}
		\Psi:& Der_K(TA,TA) & \longrightarrow & Z(A) \oplus Der_K(A,A)\oplus Der_K(A,DA)\oplus Alt_A(DA)\\
		&\left( \begin{matrix}
		\varphi_{A,A} & \varphi_{DA,A}\\
		\varphi_{A,DA} & \varphi_{DA,DA}
		\end{matrix}\right)
		&\longmapsto & (\varphi_{DA,DA}+D(\varphi_{A,A}))+\varphi_{A,A}+\varphi_{A,DA}+\varphi_{DA,A},
		\end{array}
		\end{equation*}
		
		where  $D(\varphi):DA\rightarrow DA$ with $D(\varphi_{A,A})(f)=f\circ\varphi_{A,A}$. This morphism also induces a linear isomorphism in cohomology, taking into account that $HH_1(A)^*$ is isomorphic to the first Hochschild cohomology space of $A$ with coefficients in $DA$.
		\[\overline{\Psi}: HH^1(TA) \longrightarrow  Z(A) \oplus HH^1(A)\oplus HH_1(A)^*\oplus Alt_A(DA).\]
	Furthermore, it is shown in \cite[Th\'eor\`eme 2.2.2.1]{Str} that the bracket defined by 
		\[[(a,p), (b,q)] =(b,q^{(a,p)}) -(a, p^{(b, q)}), \]
		for all $(a,p),(b,q) \in Q_1\vert \vert B$ induces an isomorphism of  Lie algebras between $\Ker d^1 / \Im d^0$ and $HH^1(A)$  where $d^0$ and $d^1$ are defined in \eqref{d0} and \eqref{d1} respectively.
As a consequence of  \cite[Th\'eor\`eme 4.1.0.6]{Str}, the following hold:
	
(1) The image of  $HH^1(A)$ with the usual bracket is a Lie subalgebra of $HH^1(TA)$. The factors $Z(A)$, $HH_1(A)^*$ and $Alt_A(DA)$ are abelian Lie subalgebras of $HH^1(TA)$.

(2) The action of the Lie algebra $HH^1(A)$ on the centre $Z(A)$ is given by: if $z\in Z(A)$, $\varphi_{A,A}\in HH^1(A)$, then \[[z,\varphi_{A,A}]=-\varphi_{A,A}(z).\] Also, $\cL:=HH^1(A)\oplus Z(A)$ is a Lie subalgebra of $HH^1(TA)$.

(3) The Lie algebra  $\cL$  acts on $HH_1(A)^* \cong H^1(A,DA)$ as follows
	\begin{eqnarray}
			\nonumber[z,\varphi_{A,DA}] & = &	z\varphi_{A,DA},\\    
			\nonumber [\varphi_{A,A},\varphi_{A,DA}]& = & -D(\varphi_{A,A})\circ\varphi_{A,DA}-\varphi_{A,DA}\circ\varphi_{A,A},
	\end{eqnarray}
	for any $z\in Z(A)$, $\varphi_{A,A}\in HH^1(A)$ and  $\varphi_{A,DA}\in H^1(A,DA)\cong HH_1(A)^*$. 

During the rest of the article we will use indistinctly the notations $HH_1(A)^*$ and $H^1(A,DA)$ depending on the context.

(4) The Lie algebra $\cL$ acts on  $Alt_A(DA)$: if $z\in Z(A)$, $\varphi_{A,A}\in HH^1(A)$ and $\varphi_{DA,A}\in Alt_A(DA)$ then,
	\begin{eqnarray}
		\nonumber[z,\varphi_{DA,A}] & = &	-z\varphi_{DA,A},\\    
		\nonumber [\varphi_{A,A},\varphi_{DA,A}]& = & \varphi_{A,A}\circ\varphi_{DA,A}+\varphi_{DA,A}\circ D(\varphi_{A,A}).
	\end{eqnarray}

(5) The  $K$-vector space $\cP:=H^1(A,DA)\oplus Alt_A(DA)\cong HH_1(A)^*\oplus Alt_A(DA)$ 
is invariant under the action of $\cL$ and satisfies $[\cP,\cP]\subset \cL$:
\[[\varphi_{A,DA}, \varphi_{DA,A}] = \varphi_{A,DA}\circ\varphi_{DA,A}-\varphi_{DA,A}\circ \varphi_{A,DA}-D(\varphi_{DA,A}\circ \varphi_{A,DA}),\]
for any $\varphi_{A,DA}\in H^1(A,DA)\cong HH_1(A)^*$ and $\varphi_{DA,A}\in Alt_A(DA)$. 

	We will use well-known comparison morphisms of $A$--bimodules between our resolution and the Bar resolution, 
	inducing the following linear morphisms, where $\varepsilon=a_n...a_1$
	$$\begin{array}{l}
	\begin{array}{llll}
	\overline{\varsigma_1}:& Hom_{K}(A,A)  &\longrightarrow& K(Q_1\vert \vert  \cB) \\
	& f: A\longrightarrow A& \longmapsto &\sum_{(a,\alpha)\in Q_1\vert \vert \cB}\lambda_{a,\alpha}(a,\alpha)\\
	&\,\,\,\,\,\,\,\,\,\, \varepsilon \longmapsto \sum_{\alpha\in \cB}\lambda_{\varepsilon,\alpha} \alpha& & \\
	&&&\\
	\overline{\omega_1}:& K(Q_1\vert \vert  \cB)  &\longrightarrow& Hom_{K}(A,A) \\
	& (a,\alpha) & \longmapsto & f_{(a,\alpha)}: A\longrightarrow A\\
	& & & \,\,\,\,\,\,\,\,\,\,\,\,\,\,\,\,\,\,\, \varepsilon \longmapsto \varepsilon^{(a,\alpha)}
	\end{array} \\
	\begin{array}{llll}
	\underline{\varsigma_1}:& Hom_{K}(A,DA)  &\longrightarrow& K(Q_1\odot  \cB^*)\\
	& g: A\longrightarrow DA& \longmapsto & {\sum_{(a, \alpha)\in Q_1\odot \cB}\lambda_{a,\alpha}(a,\alpha^*)}\\
	&\,\,\,\,\,\,\,\,\,\, \varepsilon \longmapsto \sum_{\alpha\in \cB}\lambda_{\varepsilon,\alpha} \alpha^*& & \\
	&&&\\
	\underline{\omega_1}:& K(Q_1\odot \cB^*) &\longrightarrow& Hom_{K}(A,DA) \\
	& (a,\alpha^*) & \longmapsto & g_{(a, \alpha^*)}: A\longrightarrow DA\\
	& & & \,\,\,\,\,\,\,\,\,\,\,\,\,\,\,\,\,\,\,\,\,\, \varepsilon \longmapsto \sum_{i=1}^{n}a^*(a_i)a_n...a_{i+1}\alpha^*a_{i-1}...a_1
	\end{array} \\
	\end{array}$$

\subsection{The Lie algebra structure of $HH^1(A)$}

Given a gentle algebra $A$, in this section we determine the Lie algebra structure of $HH^1(A)$. 

	We note that in general for  $(p,q)$ and $(p',q')$  in $HH^1(A)$, we have 
	\[ [(p,q),(p',q')]=(p',q'^{(p,q)})-(p,q^{(p',q')}). \]
	The non-zero brackets are the following:
	\begin{itemize}
		\item if $(a,q),(a',q')$ are shortcuts, then
		\[[(a,q),(a',q')]=\left\lbrace
		\begin{array}{cl}
		(a',a')-(a,a) & \text{if  }  a=q' \text{ and } a'=q,\\
		0 & \text{otherwise.}
		\end{array}\right.\]
		Note that $(a',a')-(a,a)=2(a',a')=-2(a,a)$ and that in fact the only case where the bracket
of two shortcuts is non-zero is when they are of the form $(a, b)$ and $(b, a)$ for the
Kronecker quiver with the two parallel arrows denoted $a$ and $b$, and, moreover, the characteristic is different from $2$.
		\item if $(b,q)$ is a shortcut and $(a,a)$ is determined by a fundamental cycle, then
		\[[(b,q),(a,a)]=\left\lbrace
		\begin{array}{cl}
		(b,q) & \text{if  }  a= b,\\
		-(b,q) & \text{if  $a$ is an arrow in } q,\\
		0 & \text{otherwise.}
		\end{array}\right. \]
	 It is worth noticing that different choices of the arrow $a$ in a fundamental cycle would not affect the Lie algebra structure since they only produce a change of sign in the bracket.
		\item if $(b,q)$ is a deviation via $b$ and $(a,a)$ is determined by a fundamental cycle, then
		\[[(b,q),(a,a)]=\left\lbrace
		\begin{array}{cl}
		-(b,q) & \text{if $a\neq b$ and $a$ is an arrow in } q ,\\
		0 & \text{otherwise.}
		\end{array}\right.\]
		The situation here regarding different choices of the arrow $a$ in a fundamental cycle is similar to the previous item.
	\item if ${\rm char }  K=2$, $(a,a)$ is determined by a fundamental cycle and $b$ is a loop then
		\[[(a,a),(b,e_j)]=\left\lbrace
		\begin{array}{cl}
		-(b,e_j) & \text{if }  a= b,\\
		0 & \text{otherwise.}
		\end{array}\right.\]
	
	\end{itemize}

We summarise the above in a table.

\begin{table}[H]
	\resizebox{15.5cm}{!}{
		\begin{tabular}{|c|c|>{\columncolor{gray!20}}c|c|c|c|}
			\cmidrule{1-6}
			
			\multicolumn{2}{|c|}{\multirow{2}{.8cm}{$[-,-]$}} & \multicolumn{4}{c|}{$HH^1(A)$}\\ 
			\cmidrule{3-6}
			
			\multicolumn{2}{|c|} {} & $(b',e_j)$  if ${\rm char }  K=2$ &$(a',a')$, $a'$ in fundamental cycle & $(b',q')$ deviation via $b'$ & $(b',q')$ shortcut \\ 
			\cmidrule{1-6}
			
			\multirow{9}{1.3cm}{$HH^1(A)$} & $(b,q)$ shortcut & $0$ &
			$\begin{array}{cl}
			(b,q) & \text{if }  b=a',\\
			-(b,q) & \text{if }  a' \text{ appears in  } q,\\
			0 & \text{otherwise.}
			\end{array}$ & 0 & 
			$\begin{array}{cl}
			2(q,q)& \text{if }  b=q' \text{ and }  q=b',\\
			0 & \text{otherwise.}
			\end{array}$\\
			\cmidrule{2-6}
			
			& $(b,q)$ deviation via $p$ &$0$& 
			$\begin{array}{cl}
			-(b,q) & \text{if }  a'\neq q \text{ and } a' \text{ appears in } q,\\
			0 & \text{otherwise.}
			\end{array}$ 
			& 0 \\
			\cmidrule{2-5}
			
			& $(a,a)$, $a$ in fundamental cycle &  $\begin{array}{cl}
			-(a, e_i) & \text{if }  a=p' \; (\&\; i=j),\\
			0 & \text{otherwise.}
			\end{array}$ &0 \\
			\cmidrule{2-4}
			\rowcolor{gray!20} \cellcolor{white} &  $ (b,e_i)$ if ${\rm char }  K=2$ & 0 \\
			\cmidrule{1-3}
		\end{tabular}
	}
\caption{Bracket in $HH^1(A)$.}
\label{table1}
\end{table}

\subsection{The Lie algebra structure of the first Hochschild cohomology space of a Brauer graph algebra}

Let $\Lambda$ be a Brauer graph algebra with multiplicity $1$ and let $A$ be a gentle algebra such that $\Lambda = TA$. We have the following brackets in $HH^1(\Lambda) = HH^1(TA)$. 

		(1) Bracket between $Z(A)$ and $HH^1(A)$: if $z\in Z(A)$ and $f_{(a,q)}$ is an element of the basis of $HH^1(A)$, by  \cite[Th\'eor\`eme 4.1.0.6 (2)]{Str}  we have  $[z,f_{(a,q)}]=-f_{(a,q)}(z)=-z^{(a,q)}$. Thus
		\begin{itemize}
			\item if $p$ is a cycle in $A$ with $\val(s(p))=2$ and $e_i=s(p)$, then
			\begin{equation*}
			[(e_i,p),(a',q')]= \left\lbrace
			\begin{array}{ll}
			-(e_i,e_i) & \text{if } p=a' \text{ is a loop with ${\rm char }  K=2$ and } q'=e_i,\\
			-(e_i,p) & \text{if } a'=q' \text{ and } a' \text{ is an arrow in } p,\\
			0 & \text{otherwise.}
			\end{array}
			\right.
			\end{equation*}
			\item if $z=1_A=\sum_{e_i\in Q_0} e_i$, then $\left[ \sum_{e_i\in Q_0} (e_i,e_i),(a',q')\right] =0$ for each $(a',q')\in HH^1(A).$
		\end{itemize} 
		
	We summarize the above in the following table
		
		\begin{table}[H]
			\resizebox{15.5cm}{!}{
				\begin{tabular}{|c|c|>{\columncolor{gray!20}}c|c|c|c|}
					\cmidrule{1-6}
					
					\multicolumn{2}{|c|}{\multirow{2}{1.1cm}{$[-,-]$}} & \multicolumn{4}{c|}{$HH^1(A)$}\\ 
					\cmidrule{3-6}	
					
					\multicolumn{2}{|c|} {}& $(a',e_j)$  if ${\rm char }  K=2$ & $(a',a')$ \, $a'$ in fundamental cycle& $(a',q')$ deviation via $a'$ & $(a',q')$ shortcut\\ 
					\cmidrule{1-6}
					
					\multirow{3}{0.8cm}{$Z(A)$}& $(e_i,p)$& $-\left( e_i,p^{(a',e_j)}\right)$ & $-\left( e_i,p^{(a',a')}\right) $ & 0 & 0\\ 
					\cmidrule{2-6}
					
					& $\sum_{e\in Q_0}(e,e)$ &$0$&$0$&$0$&$0$\\
					\cmidrule{1-6}
				\end{tabular}
			}
		\caption{Bracket between $Z(A)$ and $HH^1(A)$.}
		\label{table2}
		\end{table}
		
		(2) Bracket between $Z(A)$ and $H^1(A,DA)$: if $z\in Z(A)$ and $g\in H^1(A,DA)$, then
		$[z,g](\varepsilon) =  zg(\varepsilon),$
		hence, if $g_{(a,q^*)}$  is a summand of $g$   we have
		\[ zg_{(a, q^*)}(\varepsilon)=\sum_{i=1}^{n}a^*(a_i)za_{n}...a_{i+1}q^* a_{i-1}...a_1.\]
		Thus 
		\begin{align*}
		[z,(a,q^*)]_g& = 	 \sum_{ (c,r)\in Q_1\odot \cB} (z g_{(a,q^*)}(c))(r)(c, r^*) 
		 = 	\sum_{ (c, r)\in Q_1\odot \cB}a^*(c)(zq^*)(r)(c, r^*)\\
		& = \sum_{\{ r|(a, r)\in Q_1\odot \cB\}}q^*(r z)(a, r^*),
		\end{align*}
		where $[z,(a,q^*)]_g$ is the corresponding  summand in $[z,\underline{\varsigma_1}(g)]$. Using Theorems \ref{basisZ(A)} and \ref{basisHH_1} we complete the corresponding entries in Table \ref{table3}.
		
		(3) Bracket between $HH^1(A)$ and $H^1(A,DA)$: Let $f_{(a,q)}\in HH^1(A)$ and $g\in H^1(A,DA)$ with $g_{(a', q'^*)}$ and $[-,-]_g$ as in  $(2)$ above. Then 
		\begin{eqnarray}
		\nonumber [f_{(a,q)},g_{(a', q'^*)}]_g(\varepsilon)(\rho)& = & -\left( D(f_{(a,q)})\circ g_{(a', q'^*)}\right) (\varepsilon)(\rho)-\left( g_{(a', q'^*)}\circ f_{(a,q)}\right) (\varepsilon)(\rho)\\
		\nonumber& = & -\left(g_{(a', q'^*)} (\varepsilon)\circ f_{(a,q)}\right) (\rho)-\left( g_{(a', q'^*)}\left( f_{(a,q)}(\varepsilon)\right)\right) (\rho)\\
		\nonumber &=& -g_{(a',q'^*)} (\varepsilon) (\rho^{(a,q)})-\left( g_{(a',q'^*)}\left( \varepsilon^{(a,q)}\right)\right) (\rho).
		\end{eqnarray}
		Thus, 
		\begin{eqnarray}
		\nonumber[(a,q),(a', q'^*)]_g=	- \sum_{(b, r)\in Q_1\odot \cB}\left( ( g_{(a', q'^*)}(b))(r^{(a,q)})(b,r^*)+ (g_{(a', q'^*)}(b^{(a,q)}))(r)(b,r^*)\right) \\    
		\nonumber = -\sum_{(b, r)\in Q_1\odot \cB}a'^*(b) q'^*(r^{(a,q)})(b,r^*)- \sum_{\{r|(a, r)\in Q_1\odot \cB\}} \sum_{i=1}^{n} a'^*(q_i) q'^*(q_{i-1}...q_1r q_n...q_{i+1})(a,r^*)\\
		\nonumber =  -\sum_{\{r|(a', r)\in Q_1\odot \cB\}} q'^*(r^{(a,q)})(a',r^*)-\sum_{\{r|(a, r)\in Q_1\odot \cB\}} \sum_{i=1}^{n} a'^*(q_i) q'^*(q_{i-1}...q_1r q_n...q_{i+1})(a,r^*)
		\end{eqnarray}
		where $q=q_n ... q_1$. Using  Theorems \ref{HH^1} and \ref{basisHH_1} we  complete the calculations and record the results in Table \ref{table3}.
	
	\begin{table}[H]
		\resizebox{15.5cm}{!}{
			\begin{tabular}{|c|c|>{\columncolor{gray!20}}c|c|c|}
				\cmidrule{1-5}
				
				\multicolumn{2}{|c|}{\multirow{2}{1.1cm}{$[-,-]$}} &  \multicolumn{3}{c|}{$HH_1(A)^*$}\\ 
				\cmidrule{3-5}	
				
				\multicolumn{2}{|c|} {}& $(a', a'^*)$ if ${\rm char }  K=2$ &  $(a', e_j^*)$ & $(a', b'^*)-(b', a'^*)$\\ 
				\cmidrule{1-5}
				
				\multirow{3}{0.6cm}{$Z(A)$}& $(e_i,p)$&
				$\begin{array}{cl}
				(p, e_i^*) & \text{if }  p=a',\\ 
				0  & \text{otherwise.}
				\end{array}$
				& 0 & 0\\ 
				\cmidrule{2-5}
				
				& $\sum_{e\in Q_0}(e,e)$& $(a', a'^*)$ & $(a', e_j^*)$ & $(a', b'^*)-(b', a'^*)$\\
				\cmidrule{1-5}
				
				\multirow{7}{1.4cm}{$HH^1(A)$} & $(a,q) \mbox{ shortcut}$& 0 & 0 &0\\
				\cmidrule{2-5}
				
				& $(a,q)$ deviation via $a$ & 0 &0 & 0 \\
				\cmidrule{2-5}
				
				&$(a,a)$ $a$ in fundamental cycle& 0 & 
				$\begin{array}{cl}
				-(a, e_j^*)& \text{if }  a=a',\\
				0&  \text{otherwise.}
				\end{array}$ & 
				$\begin{array}{cl}
				-(a', b'^*)+(b', a'^*)& \text{if } a=a' \text{ or } a=b',\\
				0  &\text{otherwise.}
				\end{array}$\\
				\cmidrule{2-5}
				
				\rowcolor{gray!20}
				\cellcolor{white} & $(a,e_i)$ if ${\rm char }  K=2$ & 0 &
				$\begin{array}{cl}
				-(a, a^*) &\text{if }  i=j \text{ and } a=a',\\
				0 & \text{otherwise.}
				\end{array}$
				& 0 \\
				\cmidrule{1-5}
			\end{tabular}
		}
	\caption{$[Z(A),HH_1(A)^*]$ and $[HH^1(A),HH_1(A)^*]$.}
	\label{table3}
	\end{table}
	
		(4) Bracket between $Z(A)$ and $Alt_A(DA)$:
		\begin{itemize}
			\item Let  $\psi_{\{e_j, p'\}}, \psi_{\{p', p'\}}$ and $\phi_{\{p', q'\}}$ be as in  Theorem \ref{basisAlt}.  We identify $(e_i,p)$ with  $p \in Z(A)$. Then we have  
			\newline (1) \begin{eqnarray}
				\nonumber \left[ p, \psi_{\{e_j, p'\}}\right](q^*) & = &	-p\psi_{\{e_j, p'\}}(q^*) =-p\sum_{ \{ r|(q,r)\in\overline{(e_j,p')}\}}r\\
				\nonumber& = &-\sum_{\{r|(q,r)\in\overline{(e_j,p')}\}}pr=\left\lbrace
					\begin{array}{cl}
						-p& \text{if }  q=p \text{ and } e_i=e_j,\\
						0&  \text{otherwise.}
					\end{array}\right.
				\end{eqnarray} 
				Hence, 
				\[\left[ p, \psi_{\{e_j, p'\}}\right]=\left\lbrace
				\begin{array}{cl}
				-(p^*,p) & \text{if }  p=p',\\
				0&  \text{otherwise.}
				\end{array}
				\right.\]
				
				(2)  $[p,\psi_{\{p',p'\}}]=0$ because  $-pp'=0$ for any $p \in Z(A)\setminus\{1_A\}$.
				
				(3)  $[p,\phi_{\{p',q'\}}]=0$ because $p'\neq pq$ and $q'\neq pq$ for any $q\in A$.

			\item Let $\psi$ be an element in $Alt_A(DA)$, since $\sum_{e_i\in Q_0}e_i=1_A$,  we have \[[1_\Lambda,\psi]=-1_\Lambda\psi=-\psi.\]
		\end{itemize}
		
	(5) Bracket between $HH^1(A)$ and $Alt_A(DA)$: Let $f_{(a,q)}\in HH^1(A)$ and $\psi\in Alt_A(DA)$:
		\begin{eqnarray}
		\nonumber[f_{(a,q)},\psi](r^*)& = & (f_{(a,q)}\circ\psi) (r^*)+(\psi\circ D(f_{(a,q)}))(r^*)\\    
		\nonumber& = &f_{(a,q)}(\psi (r^*))+\psi (D(f_{(a,q)})(r^*))\\
		\nonumber& = &\psi (r^*)^{(a,q)}+\psi (r^*\circ f_{(a,q)}).	
		\end{eqnarray}
		If we write $\psi=\sum \lambda_{(p',q')}(p'^*\mapsto q')$ where $\psi(p'^*)=\lambda_{(p',q')} q'$ and $\lambda_{(p',q')}\in K$, then
		\[((p'^*\mapsto q')(r^*))^{(a,q)}=\left\lbrace \begin{array}{ll}
		q'^{(a,q)} & \text{if } r=p',\\
		0 & \text{otherwise,}
		\end{array}\right.\]
		and
		\[(p'^*\mapsto q')(r^*\circ f_{(a,q)})=\left\lbrace \begin{array}{ll}
		q' & \text{if } r=p'^{(a,q)},\\
		0 & \text{otherwise.}
		\end{array}\right. 
		\]
		Using Theorems  \ref{HH^1} and \ref{basisAlt} we complete the calculations and record the results in  Table \ref{table4}.
		
		(6) Bracket between $H^1(A,DA)$ and $Alt_A(DA)$:
		Let $g\in H^1(A,DA)$ and $\psi\in Alt_A(DA)$, we recall that
		\[[g,\psi]=g\circ\psi -D(\psi\circ g)-\psi\circ g,\]
		with $\psi\circ g \in H^1(A,DA)$  and $g\circ\psi -D(\psi\circ g ) \in Hom_{A\text{-}A}(DA,DA)\cong Z(A)$.  Similarly to the above, we write $g_{(a,q^*)}$ and $(p'^*\mapsto q')$. Now  using  Theorems \ref{basisHH_1} and \ref{basisAlt} we can calculate the brackets in Table \ref{table4}.	

\begin{table}[H]
	\resizebox{15.5cm}{!}{
		\begin{tabular}{|c|c|c|>{\columncolor{gray!20}}c| >{\columncolor{gray!20}}c|}
			\cmidrule{1-5}
			\multicolumn{2}{|c|}{\multirow{3}{1cm}{$[-,-]$}} & \multicolumn{3}{c|}{$Alt_{A}(DA)$}\\ 
			\cmidrule{3-5}	
			
			\multicolumn{2}{|c|} {}&$\phi_{\{p',q'\}}$& \cellcolor{gray!20} $\psi_{\{p',p'\}}$ if ${\rm char }K=2$ & \cellcolor{gray!20} $\psi_{\{e_j,p'\}}$ if ${\rm char }  K=2$ \\ 
			\cmidrule{1-5}
			
			\multirow{3}{.8cm}{$Z(A)$}& $(e_i,p)$& 0 & 0 &
			$\begin{array}{cl}
			-\psi_{\{p,p\}} &\text{if }  p=p',\\
			0 & \text{otherwise.}
			\end{array}$\\
			\cmidrule{2-5}
			
			& $\sum_{e\in Q_0}(e,e)$& $-\phi_{\{p',q'\}}$ & $-\psi_{\{p',p'\}}$ & $-\psi_{\{e_j,p'\}}$ \\
			\cmidrule{1-5}
			
			\multirow{7}{1.3cm}{$HH^1(A)$} & $(a,q) \mbox{ shortcut}$& 0 & 0 & 0 \\
			\cmidrule{2-5}
			
			& $(a,q)$ deviation via $a$& 0 & 0 & 0\\
			\cmidrule{2-5}
			
			& $(a,a)$ $a$ in fundamental cycle &
			$\begin{array}{cl}
			\phi_{\{p',q'\}}& \text{if } a \text{ appears in } p' \text{ or } q', \\
			0 & \text{otherwise.}
			\end{array}$ &
			0 &
			$\begin{array}{cl}
			\psi_{\{e_j,p'\}} & \text{if } a \text{ appears in } p', \\
			0 & \text{otherwise,}
			\end{array}$  \\
			\cmidrule{2-5}
			
			\rowcolor{gray!20}
			\cellcolor{white} &  $(a,e_i)$ if ${\rm char }  K=2$& 0 &
			$\begin{array}{cl}
			\psi_{\{e_i,a\}} &  \text{if } a=p', \\
			0 & \text{otherwise.}
			\end{array}$ & 0 \\
			\cmidrule{1-5}
			
			\multirow{7}{1.4cm}{$HH_1(A)^*$} & $(a, b^*)-(b, a^*)$&
			$\begin{array}{cl}
			2(a,a) &\text{if } a=p' \text{ and } b=q',\\
			-2(b,b) & \text{if } a=q'\text{ and } 	b=p',\\
			0& \text{otherwise.}
			\end{array}$ &
			0& 0\\
			\cmidrule{2-5}
			
			& $(a, e_i^*)$& 0 &
			$\begin{array}{cl}
			(e_i, a)&  \text{if } a=p', \\
			0 &\text{otherwise.}
			\end{array}$
			& 
			$\begin{array}{cl}
			(e_i,e_i)-(a, a)& \text{if } a=p' \,\, ( \& \,\, j=i),\\
			0 & \text{otherwise.}
			\end{array}$\\
			\cmidrule{2-5}
			
			\rowcolor{gray!20}
			\cellcolor{white}& $(a, a^*)$ if ${\rm char }  K=2$& 0 &
			$\begin{array}{cl}
			-(a, a)& \text{if } a=p',\\
			0 & \text{otherwise.}
			\end{array}$
			&
			$\begin{array}{cl}
			-(a, e_j)& \text{if } a=p',\\
			0 & \text{otherwise.}
			\end{array}$\\
			\cmidrule{1-5}
		\end{tabular}
	}
\caption{$[-,Alt_A(DA)]$.}
\label{table4}
\end{table}

\section{Nilpotent and solvable Lie algebras}

In this section we show that as Lie algebras the first Hochschild cohomology of gentle and Brauer graph algebras with multiplicity one are almost always solvable. A similar result for monomial algebras $KQ/I$ where $Q$ contains no cycle was obtained in  \cite[Theorem 4.23]{Str2} for $K$ of characteristic zero.

	\begin{theorem}\label{thm:nilpotent}
		Let $A$ be a gentle algebra. Then $HH^1(TA)$ is  nilpotent if and only if	
		  all of the following conditions hold:
		\begin{enumerate}
			\item $Z(A) = K$,
			\item $HH_1(A)^* = 0$,
			\item $Alt_A(DA) = 0$,
			\item for all $(p,q)$ in the basis of $HH^1(A)$,  we have $p = q$. 
		\end{enumerate}
	\end{theorem}

\begin{proof}
	Suppose that the four conditions are satisfied.
	Then $HH^1(TA)=K\oplus HH^1(A)$ and $HH^1(A)$ has a basis composed of  elements $(a,a)$ with $a$ determined by fundamental cycles. But we have already shown that the bracket of two such elements is zero, as well as their action on $Z(A)=K$.
	
	For the converse, we will prove that if one of the four conditions is not satisfied, then $HH^1(TA))$ is not nilpotent. We will show in each case how the nilpotency of $HH^1(TA)$ fails. 
	
	For	(1), if $p \in Z(A)\setminus \{1_A\}$, then there exists an arrow $a$ in the cycle $p$ such that $[(s(p),p),(a,a)]$ $=-(s(p),p)$.
	For (2), we have $[1_A, x]=x$ for any $x\in HH_1(A)^*$ and for (3), we have $[1_A, x]=-x$ for any $x\in Alt_A(DA)$. Condition (4) holds if and only if $HH^1(A)$ is given by fundamental cycles. So suppose that this is not the case. If $a$ is a shortcut for $q$, then there exists at least one fundamental cycle containing an arrow that we will call $b$. Furthermore, in this case $aq$ is not a directed cycle. Thus $[(a,q),(b,b)]=\pm(a,q)$. Now suppose that there is a deviation $q$ via $a$ where $q=q_2aq_1$.
	Let $b$ be an arrow in one of the cycles $q_1$ or $q_2$. Then $[(a,q),(b,b)]=-(a,q)$.
	Finally,  if ${\rm char }  K=2$ and there exists a loop $a$, then $(a,s(a))$ is an element in the basis of $HH^1(A)$ and $(a,s(a)^*)$ is in the basis of $HH_1(A)^*$, therefore $HH_1(A)^*\neq 0$ and the result follows. 
\end{proof}

\begin{corollary}
	Let $A=KQ/I$ be a gentle algebra with basis of paths $\cB$ and such that $Q$ contains no oriented cycles. 
	  Then $HH^1(TA)$ is nilpotent if and only if there is no arrow $a$ parallel to an element in $p \in \cB$ with $p \neq a$. 

\end{corollary}

The next result  is a direct consequence of Theorem \ref{thm:nilpotent} and Tables 1 and 2.
\begin{corollary}
	Let $A$ be a gentle algebra. If $HH^1(TA)$ is nilpotent then $HH^1(TA)$ is an abelian Lie algebra.
\end{corollary}

If the graph of $A$ is a tree, then $TA$ is of finite representation type and $HH^1(TA)$ is a nilpotent abelian Lie algebra. However the converse is not true. Namely if  $HH^1(TA)$ is a nilpotent or even an abelian Lie algebra  then $TA$  is  not necessarily of finite representation type:
\begin{Ex}{\rm 
	 Let $A$ be the gentle algebra with ideal $I=\langle ba,db \rangle$ given by the following quiver and ribbon graph.
	 
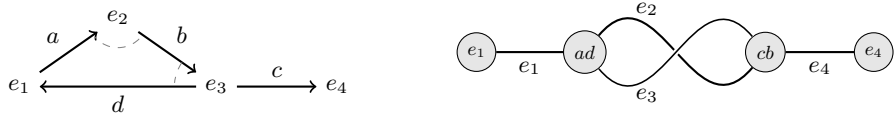
\vspace{-0.2cm}
	\begin{figure}[H]
		\centering
		\tikzstyle{help lines}+=[dashed]
			\begin{tikzpicture}[auto, thick, scale=0.8]+=[dashed]
			{\footnotesize
				\clip (-0.4,-0.5) rectangle (14.65,1.5);
				\node(1) at (0:0){$e_1$};
				\node(2) at ($(1)+(35:2)$){$e_2$};
				\node(3) at ($(2)+(-35:2)$){$e_3$};
				\node(4) at ($(3)+(0:2)$){$e_4$};
				\draw[->](1) to node   {$a$} (2);
				\draw[->](2) to node   {$b$} (3);
				\draw[->](3) to node   {$d$} (1);
				\draw[->](3) to node   {$c$} (4);		
				\draw[style=help lines] (2) +(-45:.5cm) arc (-45:-140:.5cm);
				\draw[style=help lines] (3) +(145:.7cm) arc (145:175:.7cm);
				\node[cg] (1x) at ($(4)+(0:1.5)+(35:1)$){$e_1$};
				\node[cg] (2x) at ($(1x)+(0:1.8)$){$ad$};
				\node[cg] (3x) at ($(2x)+(0:3)$){$cb$};
				\node[cg] (4x) at ($(3x)+(0:1.8)$){$e_4$};
				
				\draw[-] (1x) to node [below, pos=0.5] {$e_1$}(2x);
				\draw[-] (2x) to [out=50, in=230,looseness=1.5] node [above, pos=0.3]{$e_2$}(3x);
				\draw[draw=white,double=black] (2x) to [out=-50, in=130,looseness=1.5] node [below, pos=0.3] {$e_3$}(3x);
				\draw[-] (3x) to  node [below, pos=0.5] {$e_4$}(4x);}	
			\end{tikzpicture}
	\end{figure}
	\vspace{-0.3cm}
	We see by Theorem~\ref{thm:nilpotent} that  $HH^1(TA)\cong K \oplus K(b,b)$ is nilpotent while $TA$ is not of finite representation type. This follows from the fact that  the Brauer graph of $TA$ is not a tree.}
\end{Ex}

\begin{theorem}\label{KroneckerTh}
Let $A$ be a gentle algebra $K$-algebra, and let $B$ be the Brauer graph algebra isomorphic to $TA$. Suppose ${\rm char }  K\neq 2$. Then
		\begin{enumerate}
			\item $HH^1(A)$ is solvable if and only if $A$ is not the Kronecker algebra. 
			\item $HH^1(B)$ is solvable if and only if $B$ is not isomorphic to the trivial extension of  the Kronecker algebra.  
		\end{enumerate}
\end{theorem}

We note that if $A$ is the Kronecker algebra then $TA \cong TB$ where $B$ is the radical square zero Nakayama algebra with two simple modules. We recall the quivers of both these algebras
\vspace{-0.2cm}
	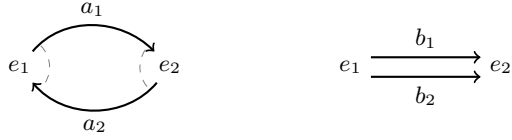
\begin{figure}[H]
		\centering
		\tikzstyle{help lines}+=[dashed]
		\begin{tikzpicture}[auto, thick, scale=0.8]+=[dashed]
		{\footnotesize
			\node(1) at (0:0){$e_1$};
			\node(1a) at ($(1)+(0:2.5)$){$e_2$};
			\draw[->] (1) to [out=50, in=135, looseness=1] node   {$a_1$} (1a);
			\draw[->] (1a) to [out=230, in=310, looseness=1]  node  {$a_2$} (1);
			\draw[style=help lines] (1) +(45:.5cm) arc (45:-45:.5cm);
			\draw[style=help lines] (1a) +(145:.5cm) arc (145:230:.5cm);
			\node(2) at ($(1a)+(0:3)$){$e_1$};
			\node (2a) at ($(2)+(0:2.5)$){$e_2$};
			\draw[transform canvas={shift={(-90:0.13)}},->] (2) to node [below] {$b_2$} (2a);
			\draw[transform canvas={shift={(90:0.13)}},->] (2) to node [above] {$b_1$} (2a);
			}
		\end{tikzpicture}
		\caption{The quiver of the radical squared zero Nakayama algebra with relations $a_1a_2=0=a_2a_1$ is on the left and the quiver of the Kronecker algebra on the right.}
	\end{figure}
\vspace{-0.2cm}
\begin{proof}
(1) If $A$ is the Kronecker algebra, then $HH^1(A)= K(b_1,b_2)\oplus K(b_2,b_1)\oplus K(b_1,b_1)$ with  the following bracket table
	\begin{table}[H]
		\resizebox{5.5cm}{!}{
			\begin{tabular}{|c||c|c|c|c|}
				\hline 
				$[-,-]$& $(b_1,b_2)$ & $(b_2,b_1)$ &  $(b_1,b_1)$\\ 
				\hline 
				\hline
				$(b_1,b_2)$ & $0$ & $-2(b_1,b_1)$ & $(b_1,b_2)$ \\ 
				\hline 
				$(b_2,b_1)$ & $2(b_1,b_1)$ & $0$ & $-(b_2,b_1)$ \\ 
				\hline 
				$(b_1,b_1)$ & $-(b_1,b_2)$ & $(b_2,b_1)$ & $0$\\ 
				\hline 
			\end{tabular}
		}
		\caption{Bracket for $HH^1(A)$ where $A$ is  the Kronecker algebra.}\label{Kronecker}
	\end{table}
Hence, $(HH^1(A))^{(1)}=HH^1(A)$.  Now,  if $A$ is not the Kronecker algebra then no element $(a,a)$ with $a$ determined by a fundamental cycle appears in $(HH^1(A))^{(1)}$ (see Table \ref{table1}), so $(HH^1(A))^{(2)}=0$.
	
(2)  If $A$ is different from the algebras mentioned in the statement, that is $A$ is neither the Kronecker algebra nor the Nakayama algebra with 2 simple modules and radical square zero relations, then the elements $1_A$ and $(a,a)$ with $a$ determined by a fundamental cycle do not appear in $\left( HH^1(TA) \right)^{(1)}$. Thus $\left( HH^1(TA) \right)^{(2)}=0$. Finally, if $A$ is the Kronecker algebra then $(HH^1(TA))^{(1)}=HH^1(A)$ and hence it is  not solvable. 
\end{proof}

	
Let $A$ be the Kronecker algebra and $B$ the Nakayama algebra with $2$ simple modules and radical square zero. As we already know $TA\cong TB$.
In the next corollary, we show that if ${\rm char} K \neq 2$ then the Lie algebra structure of $TA$ is given by $\mathfrak{gl}(2,K)$. This is noteworthy since for the Nakayama algebra $HH^1(B)$ is solvable whereas for the Kronecker algebra $HH^1(A)$ is not solvable. 
	
	\begin{corollary}
		Let $A$ be a gentle algebra with ${\rm char }  K\neq 2$. Then $HH^1(TA)$ is not solvable if and only if   $HH^1(TA)\cong \mathfrak{gl}(2,K)$. This is the case precisely if $A$ is given by the Kronecker algebra or a Nakayama algebra with 2 simples and radical square zero.  
		
		Furthermore, if $A$  is the Kronecker algebra then $HH^1(A)\cong \mathfrak{sl}(2,K)$ and if $A$ is the Nakayama algebra with 2 simples and radical square zero, then $HH^1(A) =K$. 
	\end{corollary}
	
	\begin{proof}
		If $A$ is the Kronecker algebra, consider the isomorphism between  $HH^1(A)$ and $\mathfrak{sl}(2,K)$ given by $(b_1,b_1)\mapsto \frac{h}{2}$, $(b_1,b_2)\mapsto f$, and
		$(b_2,b_1)\mapsto e$.  
		Now, since $Z(A)=K$ we have that $HH^1(TA)\cong K\oplus \mathfrak{sl}(2,K)$.
	\end{proof}
	
    	If ${\rm char} K =2$, it follows from Table~\ref{Kronecker} that $HH^1(A)$ is solvable also  in the  case that $A$ is the Kronecker algebra.

\bibliography{biblio}
\bibliographystyle{plain}



\end{document}